\documentclass{article}
\usepackage[utf8]{inputenc}
\usepackage{amsfonts}
\usepackage{amsmath}
\usepackage{autobreak}
\usepackage{amsthm}
\usepackage{enumerate}
\usepackage{mathtools}
\usepackage{appendix}
\usepackage{longtable}
\usepackage{mathrsfs}  
\usepackage{xcolor}
\usepackage{listings}
\usepackage{url}
\usepackage{authblk}
\usepackage{hyperref}
\usepackage{graphicx}

\newtheorem{theorem}{Theorem}[section]
\newtheorem{corollary}{Corollary}[theorem]
\newtheorem{lemma}[theorem]{Lemma}
\newtheorem{proposition}[theorem]{Proposition}
\newtheorem*{theorem*}{Theorem}
\newtheorem*{corollary*}{Corollary}
\newtheorem*{lemma*}{Lemma}
\newtheorem*{proposition*}{Proposition}
\theoremstyle{definition}

\newtheorem{example}[theorem]{Example}

\newtheorem*{fact*}{Fact}
\newtheorem*{example*}{Example}
\newtheorem*{remark*}{Remark}
\newtheorem*{notation*}{Notation}
\newtheorem{notation}[theorem]{Notation}
\DeclareFontShape{OT1}{cmtt}{bx}{n}{<5><6><7><8><9><10><10.95><12><14.4><17.28><20.74><24.88>cmttb10}{}
\lstdefinelanguage{Macaulay2}{
  keywords={ dummyIGuess, for, in, while, do, from, if, then, else, and, or, new,}
comment=[l]{--},
alsoletter={'},
alsoother={_},
}
\lstdefinelanguage{MAGMA}
{
    keywords={for, in, while, do, if, then, else, end, and, or, function, procedure, return, break, where},
    sensitive=true,
    morecomment=[l]{//},
    morecomment=[s]{/*}{*/},
    morestring=[b]",
}
\newcounter{magicrownumbers}
\newcommand\rownumber{\stepcounter{magicrownumbers}\arabic{magicrownumbers}}

\newcounter{magicalrownumbers}
\newcommand\rnumber{\stepcounter{magicalrownumbers}\arabic{magicalrownumbers}}

\usepackage{comment}
\usepackage{indentfirst}
\DeclareMathOperator{\Tr}{Tr}
\DeclareMathOperator{\rk}{rk}
\DeclareMathOperator{\NS}{NS}
\DeclareMathOperator{\Pic}{Pic}
\DeclareMathOperator{\T}{T}
\DeclareMathOperator{\di}{div}

\DeclareMathOperator{\Aut}{Aut}
\DeclareMathOperator{\modulo}{mod}
\DeclareMathOperator{\Homology}{H}
\DeclareMathOperator{\Isometries}{O}
\DeclareMathOperator{\Grass}{Gr}

\newcommand{\hk}{hyper-K{\"a}hler}
\newcommand{\Hk}{Hyper-K{\"a}hler}
\newcommand{\kh}{K{\"a}hler}
\newcommand{\fancyname}{good}

\newcommand{\KTST}{\hk{} fourfold of type K3$^{[2]}$}
\newcommand{\KTSTs}{\hk{} fourfolds of type K3$^{[2]}$}
\newcommand{\NeSe}{N{\'e}ron-Severi}
\newcommand{\dEPW}{double EPW sextic}
\newcommand{\dEPWs}{double EPW sextics}

\newcommand{\DEPWs}{Double EPW sextics}

\newcommand{\defeq}{\vcentcolon=}
\newcommand{\BB}{Beauville-Bogomolov}

\allowdisplaybreaks

\title{Very symmetric hyper-K\"ahler fourfolds}

\begin{document}

\author{Tomasz Wawak}
\affil{\small{Faculty of Mathematics and Computer Science, \\ Jagiellonian University}}
\date{2022}

\maketitle
\begin{abstract}
   G. H\"ohn and G. Mason classified all finite groups acting faithfully and symplectically on a \KTST{}. 
   There are 15 maximal among them, call them $\widetilde{G}_1,\ldots, \widetilde{G}_{15}$. Every manifold of type K3$^{[2]}$ admitting an action of $\widetilde{G}_i$ for some $i$ must necessarily have Picard rank 21 which is maximal. This fact allows us to use lattice-theoretic methods to classify all the finite groups $G$ acting faithfully on a \KTST{} $X$ such that $G$ contains $\widetilde{G}_i$ as a proper subgroup and $\widetilde{G}_i$ acts symplectically on $X$. We also describe examples of fourfolds of K3$^{[2]}$-type admitting an action of such groups.
\end{abstract}
\section*{Introduction}
A compact \kh{} manifold $X$ is called \textit{\hk{}} if it is simply connected and $\Homology^{2,0}(X) \cong \Homology^0(X,  \bigwedge\nolimits^2\Omega_X) = \mathbb{C} \sigma_X$ where $\sigma_X$ is a nowhere degenerate symplectic 2-form. \Hk{} manifolds are always of even dimension; K3 surfaces are \hk{} manifolds of dimension two.
\par
An automorphism of $X$ is called \textit{symplectic} if it preserves the symplectic form, otherwise it is called \textit{nonsymplectic}. 
\par
A \hk{} manifold $X$ is called a hyper-K\"ahler manifold of K3$^{[n]}$-type if it is a deformation equivalent of the Hilbert scheme of $n$ points on a K3 surface. The second cohomology group $\Homology^2(X, \mathbb{Z})$ has a natural lattice structure owing to the \BB{} form\footnote{Sometimes also called \BB -Fujiki form.} (see \cite{Beauville}). If $X$ is a K3 surface, this form induces the standard intersection pairing on the Picard group.
\par
For any finite group action $G$ on a \hk{} manifold X, one can write the following exact sequence
\begin{align} \label{the_exact_sequence}
    1 \xrightarrow{} \widetilde{G} \xrightarrow{} G \xrightarrow{} \mu_m \xrightarrow{} 1,
\end{align}
where $\mu_m$ is the group of $m$-th roots of unity for some natural $m$ and $\widetilde{G}$ is a subgroup of all the symplectic automorphisms in $G$. In the case of K3 surfaces, this simple observation along with some bounds for possible $m$ as seen in \cite{Nikulin} was one of the basic tools allowing the study of large finite groups of automorphisms of K3 surfaces after the classification of the symplectic ones by Mukai in \cite{Mu}. In
particular Kondo in \cite{Ko} shows that the maximum order of a finite group of automorphisms
of a K3 surface is $3840$ and describe such a K3 surface as a Kummer surface. In the works \cite{BS, BH} the authors describe other examples of interesting symmetric K3 surfaces. 

The goal of this paper is to follow this schema using the classification (obtained in \cite{HM}) of finite groups of symplectic automorphisms acting on \KTSTs{}. Among them, 15 are maximal.

In the sequel, we classify the polarized \KTSTs{} admitting an action of finite groups $G$ such that $\widetilde{G}$ is isomorphic to one of the 15 maximal groups; we call such manifolds \emph{very symmetric}. The classification is up to the transcendental lattice, the polarization type of the vector invariant under the action and the invariant lattice for the symplectic part of the group action; the results are presented in Table \ref{table:1}.

Section \ref{preliminary} presents some needed preliminary results of the lattice theory and the theory of \hk{} manifolds. Section \ref{construction} presents the method in which we find the transcendental lattices of very symmetric hyper-K\"ahler fourfolds and a list of these. 
Finally, Section \ref{examples} lists known explicit examples of very symmetric K3$^{[2]}$ type manifolds. Those examples are constructed as Fano varieties of lines on special cubic fourfolds as special Debarre-Voisin fourfolds or special EPW sextics.
 In particular, in Section \ref{6d2_3^4:A_6} we describe the most symmetric K3$^{[2]}$ type manifold with a group of authomorphisms of order $174960$
 as a Fano variety of lines on the Fermat cubic.
Moreover, Section \ref{squares} proves the existence of birational models of Hilbert squares of two K3 surfaces that are very symmetric.

To attain our results, we use \cite{MAGMA} and \cite{M2} systems for computations, the explanation for computations can be found in Appendix \ref{codes}. We attach the full codes as well as representations of the groups found in \cite{GAP4} format in the auxiliary files. Appendix \ref{HM_table} summarizes information on the 15 maximal groups form \cite{HM} that we use.

\subsection*{Acknowldgement}
The author would like to extend his thanks to G. Kapustka for suggesting the topic and guidance during the writing, G. H\"ohn for sharing his computational files as well as S. Mueller, G. Mongardi, M. Kapustka and Jieao Song for helpful comments. During the writing of this paper, the author has been supported by the project Narodowe Centrum Nauki 2018/30/E/ST1/00530 and received an incentive scholarship from the funds of the program Excellence Initiative - Research University at the Jagiellonian University in Kraków. Lastly, let us mention that the topic has been independently considered by Paola Comparin, Romain Demelle and Pablo Quezada Mora in a yet unpublished article; among others, they obtain a more detailed description of the \KTSTs{} admitting an action of $3^4 \colon A_6$.

\section{Preliminaries}\label{preliminary}

\subsection{Lattice theory}
We define a lattice as a free $\mathbb{Z}$-module $L$ equipped with a symmetric bilinear form
\begin{align*}
     b \colon L \times L \rightarrow \mathbb{Z}.
\end{align*}
We will often write $v \cdot w$ or $v w$ instead of $b(v, w)$ and $v^2$ instead of $b(v, v)$ for $v, w \in L$. The rank $\rk L$ is the cardinality of a basis of $L$. We define the \textit{dual} of $L$ as $L^{\vee} = \{v \in L \otimes \mathbb{Q} \colon v \cdot w \in \mathbb{Z}, \text{ for every $w \in L$} \}$ and the \textit{discriminant group} of $L$ as $D_{L} = L^{\vee}/L$. A lattice is called \emph{unimodular} if its discriminant group is trivial. A vector $v$ in $L$ is said to be of \textit{divisibility} $k$ if $v \cdot L = k \mathbb{Z}$ (then $\frac{v}{k}$ lies in $L^{\vee}$); we write $\di(v) = k$. The lattice is called \textit{even} if for any vector $v$, $v^2$ is an even number. For an even lattice, the discriminant group is equipped with the quadratic form\footnote{It is well defined because if $u$ lies in $L$, and $v$ lies in $L^{\vee}$, then $(v+u)^2 - v^2 = 2uv + u^2$ is an even integer.}
\begin{align*}
     q_{D_{L}} \colon D_{L} \ni [v] \mapsto  (v^2 \modulo 2) \in \mathbb{Q}/2\mathbb{Z},
\end{align*}
we call it the \emph{discriminant form}.

For two lattices $L_1$ and $L_2$ with bilinear forms $b_1$ and $b_2$ respectively, we call a module isomorphism $f \colon L_1 \rightarrow L_2$ an \emph{isometry} if it respects the bilinear forms i.e.\ $b_2(f(v), f(w)) = b_1(v, w)$ for any $v, w \in L_1$. We call $f$ an \emph{anti-isometry} if it $b_2(f(v), f(w)) = -b_1(v, w)$ for any $v, w \in L_1$. For a module monomorphism $g \colon L_1 \rightarrow L_2$, we call $g$ an embedding if its an isometry onto the image. An analogously, we call it an anti-embedding if it is an anti-isometry onto the image. The group $\Isometries(L)$ consists of isometries $L \rightarrow L$. 

For abelian groups equipped with quadratic forms like the ones defined for the discriminant groups we also can define isometries, anti-isometries, embeddings and anti-embeddings relative to these forms. So we define $\Isometries(D_{L})$ as the group of isometries of the discriminant group $D_L$. The isometries of $L$ naturally induce isometries of $D_{L}$. For $f \in \Isometries(L)$, we denote the isometry it induces on $D_{L}$ by $\bar{f}$. We define the subgroup $\bar{O}(L) \subset \Isometries(D_{L})$ as consisting of all the isometries of discriminant group induced by the isometries of the original lattice. 

A sublattice $M$ of $L$ is said to be \textit{primitive} if $L/M$ is torsion-free. The \emph{orthogonal complement} of $M$ is $M^{\perp} = \{ v \in L \colon v \cdot M = 0 \}$. One can show that for any sublattice of a given lattice, its orthogonal complement is primitive. 

We will often denote a lattice by a Gram matrix of its bilinear form, i.e.\ if $e_1, e_2, \ldots, e_n$ is a basis of the lattice, it will be the matrix $(e_i \cdot e_j)_{i,j = 1}^n$. Note that to avoid confusion throughout the paper we will opt to only write matrices with parentheses "( )" to denote the form and its lattice; the square brackets "[ ]" will encompass matrices representing linear transformations. By $\langle k \rangle$ (where $k$ is an integer) we will denote the rank 1 lattice with generator $v$ such that $v^2 = k$. For a lattice $A$ and an integer $m$, $A(m)$ is the lattice with the same underlying $\mathbb{Z}$-module and a bilinear form $v \cdot u = m (v \cdot_A u)$ for $u,v$ vectors in $A$ ($\cdot_A$ is the bilinear form on $A$).

A lattice is called positive (negative) definite, indefinite, hyperbolic, nondegenerate etc\ if its bilinear form has the respective property. Henceforth, all lattices are assumed to be nondegenerate.

For a lattice $L$ and the group $G$ acting on $L$ via isometries, we define the \emph{invariant sublattice} (sometimes also called the \emph{invariant lattice}).
\begin{align*}
    L^G = \{v \in L \colon g(v) = v, \text{ for any $g \in G$} \}, 
\end{align*}
and the \emph{coinvariant sublattice} (sometimes also called the \emph{coinvariant lattice}).  
\begin{align*}
    L_G = (L^G)^\perp.
\end{align*}
By simple linear algebra, both are primitive.

In the theory of K3 surfaces, particularly in the study of their automorphisms, the following lattice-theoretic result turns out to be very useful (see e.g.\ \cite[Chapter 14, Proposition 0.2]{huybrechts_2016}):
\begin{proposition}\label{isomorphism}
Let $L$ be a even unimodular lattice, $M \subset L$ a primitive sublattice, and put $N = M^{\perp}$. Then
\begin{align*}
    L / (M \oplus N) \subset M^{\vee}/M \oplus N^{\vee} / N = D_{M} \oplus D_{N}  
\end{align*}
is the graph of a an anti-isometry $\gamma \colon D_{M} \rightarrow D_{N}$. In particular, $D_M$ and $D_N$ are isomorphic as groups.
\end{proposition}
One can easily show (e.g.\ by the proof of \cite[Chapter 14, Proposition 0.2]{huybrechts_2016}) that in general the following holds:
\begin{proposition}\label{partial_monomorphism}
Let $L$ be an even lattice, $M \subset L$ a primitive sublattice, and put $N = M^{\perp}$. Then
\begin{align*}
    L / (M \oplus N) \subset M^{\vee}/M \oplus N^{\vee} / N = D_{M} \oplus D_{N}  
\end{align*}
is the graph of an anti-isometry $\gamma \colon A \rightarrow B$ for some subgroups $A \subset D_{M}$ and $B \subset D_{N}$. 
\end{proposition}

The morphism $\gamma$ above is sometimes called a gluing morphism. A more specific result in this spirit will turn out to be of use to us.
\begin{lemma} \label{extension_condition}
Let $M$ be an even lattice, $N$ its primitive sublattice, such that there exists an anti-embedding
\begin{align*}
        \gamma \colon D_{N^{\perp}} \rightarrow D_{N},
\end{align*}
and $M/(N^{\perp} \oplus N)  \subset D_{N^{\perp}} \oplus D_{N}$ is its graph. Then for an isometry $f \in \Isometries(N)$, $f$ extends to an isometry of $M$ if and only if 
\begin{enumerate}
\item $\bar{f}$ leaves $\gamma(D_{N^{\perp}})$ invariant,
\item $\gamma^{-1} \circ \bar{f} \circ \gamma$ (it is well defined by the above) is induced by an isometry of $N^{\perp}$.
\end{enumerate}
\end{lemma}
\begin{proof}
Indeed, assume there exists $\widetilde{f} \in \Isometries(M)$, such that $\widetilde{f}_{|N} = f$. Take $g = \widetilde{f}_{|N^{\perp}}$. Then $\bar{f} \oplus \bar{g}$ is an automorphism of $D_{N \oplus N^{\perp}} = D_{N} \oplus D_{N^{\perp}}$. Any element of $M/(N^{\perp} \oplus N)$ can be written as $\gamma(x) + x$ where $x$ is an element of  $D_{N^{\perp}}$. We must have  $\bar{f} \oplus \bar{g}(\gamma(x) + x) = \bar{f}(\gamma(x)) + \bar{g}(x)$ still in $M/(N^{\perp} \oplus N)$. So $\gamma \circ \bar{g} = \bar{f} \circ \gamma$, and $\bar{f}$ indeed leaves the image of $\gamma$ invariant, which means we can invert $\gamma$ in the last equality which concludes the proof in one direction. In the other one, we simply define $g$ to be an isometry which induces $\gamma^{-1} \circ \bar{f} \circ \gamma$ and $f \oplus g$ extends to an isometry of $M$.
\end{proof}

\subsection{\Hk{} manifolds and their lattices}
Let us briefly recall some theory on lattices of \hk{} manifolds (see \cite{D, Huybrechts2} for more details). For a \hk{} manifold $X$, the \BB{} form endows the second cohomology group $\Homology^2(X, \mathbb{Z})$ with a lattice structure of signature  (3, $b_2(X) - 3$) where $b_2(X)$ is the second Betti number of the manifold $X$. The N{\'e}ron-Severi group (or lattice) of $X$ is defined as
\begin{align*}
    \NS_X = \Homology^2(X, \mathbb{Z}) \cap \Homology^{1,1}(X) \subset \Homology^2(X, \mathbb{C}).
\end{align*}

The transcendental lattice $\T_X$ is the orthogonal complement of $\NS_X$ in $\Homology^2(X, \mathbb{Z})$. In the projective case, the N{\'e}ron-Severi group is isomorphic to the Picard group $\Pic(X)$. Then the induced bilinear form is hyperbolic (that is, of signature (1, $k$) for some $k \in \mathbb{N}$), also, $h^2 > 0$ for any ample class $h \in \Pic(X) = \NS_X$. For a divisor $x$ in the \NeSe{} group, we call the number $x^2$ a Beauville-Bogomolov degree of $x$.

By $\Aut(X)$ we will denote the group of biholomorphic automorphisms (we will just call them automorphisms from now on) of $X$. $\Aut_s(X)$ will be the subgroup of $\Aut(X)$ consisting of the symplectic automorphisms. 

We note the following consequence of the existence of the exact sequence (\ref{the_exact_sequence}).
\begin{lemma}
A finite simple subgroup $G \subset \Aut(X)$ for a \hk{} manifold $X$ is either cyclic or is contained in $\Aut_s(X)$.
\end{lemma}\label{simple_lemma}
\begin{proof}
    By the exact sequence (\ref{the_exact_sequence}), $G$ has a normal subgroup $\widetilde{G}$ such that $G/\widetilde{G}$ is a finite cyclic group and action of $\widetilde{G}$ on $X$ is symplectic. But by simplicity of $G$, $\widetilde{G} = G$ or $\widetilde{G}$ is trivial. 
\end{proof}

Let us observe that analyzing finite group actions on a projective \hk{} manifold $X$ is the same as analyzing groups of polarized morphisms (i.e.\ morphisms preserving the polarization) of $(X, h)$ for some ample $h \in \Pic(X)$ as a consequence of the following result (cf. \cite[Proposition 4.1]{D}).
\begin{proposition}\label{finite_kernel}
For a projective \hk{} manifold, the map
\begin{align*}
    \Psi \colon \Aut(X) \ni f \mapsto f^{*} \in \Isometries(\Pic(X))
\end{align*}
has a finite kernel.
\end{proposition}
\begin{corollary}\label{finite_group}
If $X$ is a projective \hk{} manifold, $G \subset \Aut(X)$ is a subgroup, then $G$ is finite if and only if it fixes an ample class on $X$.
\end{corollary}

\begin{proof}
If $G$ is finite, then let $h \in \Pic(X)=\NS_X$ be any ample class. Then the class $$\widetilde{h}=\sum_{g \in G}g^{*}h$$ is still ample and is invariant under $G$.

Now assume $G$ fixes an ample class $h$. Let $\Psi$ be as in Proposition \ref{finite_kernel}, and put $\Psi_{G} = \Psi_{|G}$, so $\ker \Psi_G$ is finite. Now define the quotient $\widetilde{G} = G/\ker \Psi_G$ and note that $\widetilde{G}$ acts faithfully on $\NS_X$. Take $N$ to be the orthogonal complement of $h$ in $\NS_X$. Since $\NS_X$ is of signature $(1,k)$ for some integer $k$ and $h^2 > 0$ as it is ample, then $N$ is negative definite, and therefore it has finitely many isometries (as for any integer $z$, there are only finitely many vectors $v \in N$ such that $v^2 = z$). But any isometry $f \in \Isometries(\NS_X)$ which fixes $h$ is uniquely determined by its restriction to $N$. In conclusion, $|\widetilde{G}| \le |\Isometries(N)|$ which means that $\widetilde{G}$ is finite. This implies that $G$ is also finite as $|G| = |\ker \Psi_G| \cdot |\widetilde{G}|$. 
\end{proof}

\cite[Proposition 4.1]{D} also contains the following:
\begin{proposition}\label{empty_kernel}
For a \hk{} manifold of type K3$^{[n]}$ $X$, the morphism
\begin{align*}
    \Psi \colon \Aut(X) \ni f \mapsto f^{*} \in \Isometries(\Homology^2(X, \mathbb{Z}))
\end{align*}
is injective.
\end{proposition}
Because of this, in the sequel we will identify actions of groups on any \KTST{} with the induced actions on the isometries of their integral second cohomology groups.

Now assume $X$ is a projective \KTST{}. As mentioned earlier, $\Homology^2(X, \mathbb{Z})$ has a lattice structure, it turns out to be isometric to the even lattice\footnote{$U$ is the \textit{hyperbolic plane}, a rank 2 lattice with the bilinear form $\begin{pmatrix} 0 & 1 \\ 1 & 0 \end{pmatrix}$, and $E_8$ is the lattice corresponding to the Dynkin diagram; we adopt the convention that $E_8$ is positive definite.} $L_{\text{K3}^{[2]}} = E_{8}(-1)^{\oplus 2} \oplus U^{\oplus 3} \oplus \langle -2 \rangle$.

\par
Now we are ready to state the first Torelli type theorem relevant for us (for simplicity's sake we will only state it for \KTST{}).
\begin{theorem}\label{Torelli_morphisms}
Let $(X, h)$ and $(X', h')$ be two polarized \KTSTs{}, and let
\begin{align*}
    \varphi \colon \Homology^2(X', \mathbb{Z}) \rightarrow \Homology^2(X, \mathbb{Z}),
\end{align*}
be an isometry of lattices such that $\varphi(h') = h$ and $\varphi_{\mathbb{C}}(\Homology^{2,0}(X')) = \Homology^{2,0}(X)$ where $\varphi_{\mathbb{C}}$ is the extension of $\phi$ to the automorhpism of $\Homology^2(X', \mathbb{C}) = \Homology^2(X', \mathbb{Z}) \otimes \mathbb{C}$. Then there exists a biholomorphic mapping $f \colon X \rightarrow X'$ such that $\varphi = f^{*}$.
\end{theorem}
\begin{proof}
    \cite[Theorem 1.3]{Markman}.
\end{proof}

There is another Torelli type theorem important to our discussion. Let us pick a primitive element $h \in L_{\text{K3}^{[2]}}$ of positive square. Let us define the manifold
$\Omega_h = \{ [x] \in \mathbb{P}(L_{\text{K3}^{[2]}} \otimes \mathbb{C}) \, | \, x \cdot h = x^2 = 0, \, x \cdot \bar{x} > 0\}$ and the subgroup of the isometry group $\Hat{O}(L_{\text{K3}^{[2]}}, h) = \{\varphi \in \Isometries(L_{\text{K3}^{[2]}}): \, \varphi(h) = h, \, \bar{\varphi} = \pm \text{Id} \colon D_{L_{\text{K3}^{[2]}}} \to D_{L_{\text{K3}^{[2]}}} \}$. $\Hat{O}(L_{\text{K3}^{[2]}}, h)$ acts on $\Omega_h$ which allows us to define the \emph{period domain} $\mathscr{P}_h = \Omega_h/\Hat{O}(L_{\text{K3}^{[2]}}, h)$. An element of $\mathscr{P}_h$ is called a \emph{period}. 

Let $\mathscr{M}_{2n}^{(\gamma)}$ be the moduli space of polarized \KTSTs{} with polarization of \BB{} degree $2n$ ($n > 0$) and divisibility $\gamma \in \{1,2\}$. There exists a \emph{period map} 
\begin{align*}
    \wp_h \colon \mathscr{M}_{2n}^{(\gamma)} \ni X \mapsto [\mathbb{C} \eta^{-1}(\sigma_X)] \in \mathscr{P}_h
\end{align*}
which sends an $h$-polarized manifold $X$ onto the class of the inverse image of its symplectic form by the marking $\eta \colon \Homology^2(X, \mathbb{Z}) \to L_{\text{K3}^{[2]}}$. 
\begin{theorem}\label{Torelli_period}
    The moduli space $\mathscr{M}_{2n}^{(\gamma)}$ is irreducible. The period map 
    \begin{align*}
        \wp_h \colon \mathscr{M}_{2n}^{(\gamma)} \rightarrow \mathscr{P}_h
    \end{align*}
    is injective. The image $\wp_h(\mathscr{M}_{2n}^{(\gamma)})$ consists of periods $p$ in $\mathscr{P}_h$ for which there are no vectors $v \in L_{\text{K3}^{[2]}}$ such that for a choice (hence any choice) of $\sigma \in L_{\text{K3}^{[2]}} \otimes \mathbb{C}$ of class $p$ in $\mathscr{P}_h$:
    \begin{enumerate}[(i)] \label{a}
        \item $h \cdot v = 0$,
        \item $v \cdot \sigma = v \cdot \bar{\sigma} = 0$,
        \item $v^2 = -10$, and $\di(v) = 2$; or $v^2 = -2$.
    \end{enumerate}
    \end{theorem}
\begin{proof}
    \cite[Theorem 3.22 (and its proof) and Remark 3.33]{D}. 
\end{proof}
Divisors satisfying the last condition from the lemma (a divisor also automatically will satisfy the second one) are usually called wall divisors.

\par

\section{Finding automorphisms} \label{construction}
Lattice $L_{\text{K3}^{[2]}}$ is not unimodular, but it is close to being one, as a direct sum of a unimodular lattice and a vector of square $-2$, so in the situation from Proposition \ref{partial_monomorphism} we will always get a full anti-embedding in one direction. In particular, as shown in \cite{HM} (Theorem 7.1, 8.3)
\begin{lemma} \label{monomorphism}
Let $X$ be a \KTST{}, $\widetilde{G}$ a finite subgroup of $\Aut_s(X)$. Put $L \defeq \Homology^2(X, \mathbb{Z})$. Then $\widetilde{G}$ acts on $L$ and there exists a canonical anti-embedding
\begin{align*}
    \gamma \colon D_{L_{\widetilde{G}}} \rightarrow D_{L^{\widetilde{G}}}.
\end{align*}
Moreover, the image of $\gamma$ is the subgroup of signature 2 in $D_{L^{\widetilde{G}}}$.
\end{lemma}

One can construct a natural isomorphism $D_{L} \cong D_{L^{\widetilde{G}}}/D_{L_{\widetilde{G}}}$ in the above setting using the map $\gamma$ and projections from $L^{\vee}$ to $(L^{\widetilde{G}})^{\vee}$ and $(L_{\widetilde{G}})^{\vee}$\footnote{Strictly speaking, these are projections from $(L^{\widetilde{G}})^{\vee} \oplus (L_{\widetilde{G}})^{\vee}$ restricted to $L^{\vee}$}.

The following simple lemma will be of crucial importance for us.
\begin{lemma} \label{trace_lemma}
Let $X$ be a projective \KTST{}, $L = \Homology^2(X, \mathbb{Z})$, $\widetilde{G}$ a finite group acting faithfully and symplectically on $X$\footnote{As always, we identify actions of $\widetilde{G}$ on both manifold and the lattice.} such that $\rk L^{\widetilde{G}} = 3$. Let $\hat{f}$ be a finite nonsymplectic automorphism of $X$. Then $f$, the induced action of $\hat{f}$ on $L^{\widetilde{G}}$, has finite order $k$ for some integer $k > 1$ and hence has eigenvalues 1, $\zeta_k$, and $\overline{\zeta_k}$ where $\zeta_k$ is a primitive $k$-th root of unity. In particular:
\begin{enumerate}[(i)]
\item if $f$ has order 2, then $\Tr f = -1$, furthermore $f_{|T_x} = -\text{id}_{T_X}$;
\item if $f$ has order 3, then $\Tr f = 0$;
\item if $f$ has order 4, then $\Tr f = 1$;
\item if $f$ has order 6, then $\Tr f = 2$.
\end{enumerate}
These are all the possibilities. Moreover, each such morphism $f$ preserves the isotropic lines\footnote{I.~e.\ lines on which the induced quadratic form is 0.} in $L^{\widetilde{G}} \otimes \mathbb{C}$ .
\end{lemma}
\begin{proof}
Denote the map induced on $L^{\widetilde{G}} \otimes \mathbb{C}$ again by $f$. We have $f(\sigma_X) = \zeta_k \cdot \sigma_X$ where $\zeta_k$ is a primitive root of unity of order $k$, and therefore also $f(\overline{\sigma_X}) = \overline{\zeta_k} \cdot \overline{\sigma_X}$.  Now notice that there exists a nonzero element of $L \otimes \mathbb{C}$ that is invariant for the whole $\langle \widetilde{G}, f \rangle$ (a \kh{} class $\sum_{g \in \langle \widetilde{G}, f \rangle} g( \kappa$), where $\kappa$ is any \kh{} class).
\par Now for the last assertion, $W = L^{\widetilde{G}}_{\langle f \rangle} \otimes \mathbb{C}$ is a 2-dimensional space with an induced positive definite form, so it has exactly two isotropic lines. If $f$ is of order 2, then it is minus identity on $W$ by above discussion, so it preserves the lines. If $f$ does not preserve the lines, then it has to map one into the other, so since it is finite, it must be of even order, that settles the case of order 3. Order 2 and 3 cases together settle order 6 case. Now for $f$ of order 4, $f^2$ must be minus identity since it has order 2. Then for a vector $w \in W$ we have $w \cdot f(w) = f(w) \cdot f^2(w) = - w \cdot f(w)$, so it is 0. But since the bilinear form is positive definite and $w^2 = (f(w))^2 = 0$, it must mean that $w$ and $f(w)$ lie on a single line (otherwise the form would be identically 0).
\end{proof}
\begin{notation}\label{good_notation}
We will call an isometry of a rank 3 lattice \textit{\fancyname{}} if it has order 2, 3, 4 or 6 and satisfies the respective condition from Lemma \ref{trace_lemma}.
\end{notation}

Consider a group $\widetilde{G}$ which is one of the 15 maximal finite groups which act on some \KTST{} via symplectic automorphisms. Our goal is to analyze all finite groups $G$ containing $\widetilde{G}$ as a subgroup for which there exists a \KTST{} X, such that $G$ has a representation as a subgroup of $\Aut(X)$ and $\widetilde{G}= G \cap \Aut_s(X)$ under this representation.

So let us fix $X$, a projective \KTST{}, and assume $\widetilde{G} \subset \Aut_s(X)$ as a subgroup. Put $L = \Homology^2(X, \mathbb{Z})$. From \cite{HM}, we know what $L^{\widetilde{G}}$ and $L_{\widetilde{G}}$ might look like; $L_{\widetilde{G}}$ is fixed if we know $\widetilde{G}$, but there might be multiple options for $L^{\widetilde{G}}$. Nevertheless, they are all known. So let us fix one of them. In particular, we always have $\rk L^{\widetilde{G}} = 3$, $\rk L_{\widetilde{G}} = 20$ (cf. \cite{HM}, Table 9), and $L_{\widetilde{G}}$ contains no wall divisors (cf. \cite{HM}, Theorem 8.3). We also have the anti-embedding $\gamma \colon D_{L_{\widetilde{G}}} \rightarrow D_{L^{\widetilde{G}}}$ as in Lemma \ref{monomorphism}.

Let us assume moreover that there exists $\hat{f} \in \Aut(X) \setminus \widetilde{G}$ such that $G$, the group generated by $\hat{f}$ and $\widetilde{G}$, is finite. By Lemma \ref{trace_lemma} and its proof, $f = \hat{f}_{|L^{\widetilde{G}}}$ is \fancyname{}, and by knowing how it acts on $L^{\widetilde{G}}$, we know the transcendental lattice as well as the primitive ample class fixed by $G$ (there is necessarily exactly one such class).
\subsection{The construction}
Let $L = L_{\text{K}3^{[2]}}$ (treated as an abstract lattice). Our procedure is going to be as follows. We fix $\widetilde{G}$, one of the 15 maximal groups and consider its action on  $L$ such that $L^{\widetilde{G}}$ and $L_{\widetilde{G}}$ are as in \cite{HM}. Consider abstract lattices $M$ isomorphic to  $L_{\widetilde{G}}$ and $N$ isomorphic to $L^{\widetilde{G}}$. Let us fix a \fancyname{} isometry $f$ in $\Isometries(N)$.

We want to check whether the orthogonal sum $N \oplus M$ can be embedded in $L$ in such a way that $f$ extends to an isometry of $L$. We do so by considering all anti-embeddings $\gamma \colon D_M \rightarrow D_N$ such that $\gamma$ is the gluing morphism for some embedding $N \oplus M \hookrightarrow L$ and checking the conditions from Lemma \ref{extension_condition}. The computational details of that can be found in Appendix \ref{main_code}.

If $f$ can be extended to an isometry of $L$, we call this isometry $\widetilde{f}$. Define $G$ as the group generated by $\widetilde{G}$ and $\widetilde{f}$ in $\Isometries(L)$. Because $f$ was good, $G$ fixes a unique (up to sign) primitive vector $H$ in $L$. Let $T$ be the orthogonal complement of $H$ in $N$. Now we can show the existence of a polarized manifold $(X, h)$ with a polarized action of $G$. 

Assuming we have a polarized \KTST{} $(X, h)$ and a marking $\eta \colon X \rightarrow L$ such that $\eta^{-1}(H) = h$ and $\eta^{-1}(T) = \T_X$, then through $\eta$ we have an action of $G$ on $\Homology^2(X, \mathbb{Z})$. By construction, this action fixes $h$ as the action of $G$ on $L$ fixed $H$. $\widetilde{G}$ fixes the entire $\T_X$, so in particular the action linearly extended to $\Homology^2(X, \mathbb{C})$ fixes the symplectic form $\sigma_X$; $\hat{f}$ extended to $\Homology^2(X, \mathbb{C})$ fixes $\mathbb{C}\sigma_X$ by Lemma \ref{trace_lemma} since $f$ is \fancyname{}. So the entire $G$ acts faithfully on $X$ by Theorem \ref{Torelli_morphisms}. 

But we can find such $X$ by Theorem \ref{Torelli_period} 
as $M \equiv L_{\widetilde{G}}$ which is the orthogonal complement of $h$ and $T$ in $L$ contains no wall divisors (as mentioned before cf. \cite{HM}, Theorem 8.3).

The results of the procedure applied over all the groups and their isomorphisms are in the following table. Each row describes a set of invariants of a triple $(X, h, G)$ where $(X, h)$ is a polarized \KTST{}, $G$ is a subgroup of $\Aut(X)$ fixing $h$ (hence $G$ is finite by Corollary \ref{finite_group}) such that $\widetilde{G} = G \cap \Aut_s(X)$ is one of the 15 maximal groups of symplectic automorphisms acting on the \KTSTs{}. The integer $h^2$ is the usual \BB{} degree of $h$, div denotes the divisibility of $h$.
The integer $m$ is the order of the cyclic group $G/\widetilde{G}$, $\T_X$ is the transcendental lattice of $X$. Column K3 tells if $X$ is birational to $S^{[2]}$ for some $S$ a K3 surface ("-" it is apparently not known yet, "f" means false)\footnote{We determine this by finding a vector in $\T_X$ which is of divisibility 2 in $\Homology^2(X, \mathbb{Z}$), see more in Appendix \ref{main_code}}. $L^{\widetilde{G}}$ is the invariant sublattice of the lattice $\Homology^2(X, \mathbb{Z})$ for the action of $\widetilde{G}$. "Ex" gives a reference to an example of $(X, h, G)$ in Section \ref{examples} if one is known.

Let us remark that the groups are maximal in the following sense: there exists a polarized \KTST{} $(X,h)$ such that $G$ is the group of all the $h$-polarized automorphisms on $X$. In particular, there are examples of two rows where all the invariants in the table match and a group $G$ from one row can be realized as a subgroup of a group $G'$ from another, but there are triples $(X, h, G)$ and $(X', h', G')$ such that an action of $G$ cannot be extended to an $h$-polarized action of $G'$. Note also that the triples $(X, h, G)$ do not necessarily contain one isomorphism class, i.~e. there can be two triples $(X, h, G)$ and $(X', h', G)$ such that $(X, h)$ and $(X', h')$ are not isomorphic. The notation for groups is the same as in \cite{HM} and \cite{Wilson1985ATLASOF}.

\footnotesize
\begin{center}
    \begin{longtable}{ |c|c|c|c|c|c|c|c|c|}
    \hline\hline
    
    \# & $h^2$ & div & $\widetilde{G}$ & $m$ & $\text{T}_X$ & K3 & $L^{\widetilde{G}}$ & Ex \\ 
    
    \hline\hline
    
    \rownumber. & 2 & 1 & $L_2(11)$ & 2 & $\begin{pmatrix}
    22 & 0 \\
    0 & 22
    \end{pmatrix}$
     & f & $\begin{pmatrix}
    2 & 1 & 0 \\
    1 & 6 & 0 \\
    0 & 0 & 22
    \end{pmatrix}$
     &  \ref{EPW_L2(11)}  \\ 
    \hline
    
    \rownumber. & 2 & 1 & $L_3(4)$ & 2 & $\begin{pmatrix}
    10 & 4 \\
    4 & 10
    \end{pmatrix}$
     &  -  & $\begin{pmatrix}
    2 & 0 & 0 \\
    0 & 10 & 4 \\
    0 & 4 & 10
    \end{pmatrix}$
     &  - \\ 
    \hline
    
    \rownumber. & 2 & 1 & $A_7$ & 2 & $\begin{pmatrix}
    6 & 0 \\
    0 & 70
    \end{pmatrix}$
     & f & $\begin{pmatrix}
    2 & 1 & 0 \\
    1 & 2 & 0 \\
    0 & 0 & 70
    \end{pmatrix}$
     &  \ref{EPW_A7}\footnote{Actually we do not know whether our examples correspond to this row, the one below, or one per each.} \\ 
    \hline
    
    \rownumber. & 2 & 1 & $A_7$ & 2 & $\begin{pmatrix}
    6 & 0 \\
    0 & 70
    \end{pmatrix}$
     & f & $\begin{pmatrix}
    2 & 0 & 1 \\
    0 & 6 & 0 \\
    1 & 0 & 18
    \end{pmatrix}$
     &  \ref{EPW_A7}\footnote{See above.} \\ 
    \hline
    
    \rownumber. & 2 & 1 & $\mathbb{Z}_2 \times L_2(7)$ & 2 & $\begin{pmatrix}
    14 & 0 \\
    0 & 14
    \end{pmatrix}$
     &  -  & $\begin{pmatrix}
    2 & 0 & 0 \\
    0 & 14 & 0 \\
    0 & 0 & 14
    \end{pmatrix}$
     &  - \\ 
    \hline
    
    \rownumber. & 2 & 1 & $\mathbb{Z}_2 \times L_2(7)$ & 4 & $\begin{pmatrix}
    14 & 0 \\
    0 & 14
    \end{pmatrix}$
     &  -  & $\begin{pmatrix}
    2 & 0 & 0 \\
    0 & 14 & 0 \\
    0 & 0 & 14
    \end{pmatrix}$
     &  \ref{2_2L2(7)} \\ 
    \hline
    
    \rownumber. & 2 & 1 & $\mathbb{Z}_2:A_6$ & 2 & $\begin{pmatrix}
    4 & 0 \\
    0 & 24
    \end{pmatrix}$
     &  -  & $\begin{pmatrix}
    2 & 0 & 0 \\
    0 & 4 & 0 \\
    0 & 0 & 24
    \end{pmatrix}$
     &  - \\ 
    \hline
    
    \rownumber. & 2 & 1 & $\mathbb{Z}_2^4:S_5$ & 2 & $\begin{pmatrix}
    4 & 0 \\
    0 & 40
    \end{pmatrix}$
     &  -  & $\begin{pmatrix}
    2 & 0 & 0 \\
    0 & 4 & 0 \\
    0 & 0 & 40
    \end{pmatrix}$
     &  \ref{2_Z5_S5}  \\ 
    \hline
    
    \rownumber. & 2 & 1 & $M_{10}$ & 2 & $\begin{pmatrix}
    4 & 0 \\
    0 & 30
    \end{pmatrix}$
     &  -  & $\begin{pmatrix}
    2 & 0 & 0 \\
    0 & 4 & 0 \\
    0 & 0 & 30
    \end{pmatrix}$
     &  - \\ 
    \hline
    
    \rownumber. & 2 & 1 & $M_{10}$ & 2 & $\begin{pmatrix}
    4 & 0 \\
    0 & 30
    \end{pmatrix}$
     & f & $\begin{pmatrix}
    2 & 0 & 0 \\
    0 & 4 & 0 \\
    0 & 0 & 30
    \end{pmatrix}$
     &  - \\ 
    \hline
    
    \rownumber. & 4 & 1 & $L_3(4)$ & 2 & $\begin{pmatrix}
    12 & 0 \\
    0 & 14
    \end{pmatrix}$
     & f & $\begin{pmatrix}
    4 & 2 & 0 \\
    2 & 4 & 0 \\
    0 & 0 & 14
    \end{pmatrix}$
     &  - \\ 
    \hline
    
    \rownumber. & 4 & 1 & $\mathbb{Z}_2^3:L_2(7)$ & 2 & $\begin{pmatrix}
    6 & 2 \\
    2 & 10
    \end{pmatrix}$
     &  -  & $\begin{pmatrix}
    4 & 0 & 0 \\
    0 & 6 & 2 \\
    0 & 2 & 10
    \end{pmatrix}$
     &  - \\ 
    \hline
    
    \rownumber. & 4 & 1 & $\mathbb{Z}_2 \times L_2(7)$ & 2 & $\begin{pmatrix}
    14 & 0 \\
    0 & 28
    \end{pmatrix}$
     & f & $\begin{pmatrix}
    4 & 2 & 0 \\
    2 & 8 & 0 \\
    0 & 0 & 14
    \end{pmatrix}$
     &  - \\ 
    \hline
    
    \rownumber. & 4 & 1 & $\mathbb{Z}_2:A_6$ & 2 & $\begin{pmatrix}
    2 & 0 \\
    0 & 24
    \end{pmatrix}$
     &  -  & $\begin{pmatrix}
    2 & 0 & 0 \\
    0 & 4 & 0 \\
    0 & 0 & 24
    \end{pmatrix}$
     &  - \\ 
    \hline
    
    \rownumber. & 4 & 1 & $\mathbb{Z}_2:A_6$ & 2 & $\begin{pmatrix}
    6 & 0 \\
    0 & 8
    \end{pmatrix}$
     & f & $\begin{pmatrix}
    4 & 0 & 0 \\
    0 & 6 & 0 \\
    0 & 0 & 8
    \end{pmatrix}$
     &  - \\ 
    \hline
    
    \rownumber. & 4 & 1 & $\mathbb{Z}_2^4:S_5$ & 2 & $\begin{pmatrix}
    2 & 0 \\
    0 & 40
    \end{pmatrix}$
     &  -  & $\begin{pmatrix}
    2 & 0 & 0 \\
    0 & 4 & 0 \\
    0 & 0 & 40
    \end{pmatrix}$
     &  - \\ 
    \hline
    
    \rownumber. & 4 & 1 & $\mathbb{Z}_2^4:S_5$ & 2 & $\begin{pmatrix}
    8 & 0 \\
    0 & 10
    \end{pmatrix}$
     &  -  & $\begin{pmatrix}
    4 & 0 & 0 \\
    0 & 8 & 0 \\
    0 & 0 & 10
    \end{pmatrix}$
     &  - \\ 
    \hline
    
    \rownumber. & 4 & 1 & $S_6$ & 2 & $\begin{pmatrix}
    12 & 0 \\
    0 & 30
    \end{pmatrix}$
     & f & $\begin{pmatrix}
    4 & 2 & 0 \\
    2 & 4 & 0 \\
    0 & 0 & 30
    \end{pmatrix}$
     &  - \\ 
    \hline
    
    \rownumber. & 4 & 1 & $M_{10}$ & 2 & $\begin{pmatrix}
    2 & 0 \\
    0 & 30
    \end{pmatrix}$
     &  -  & $\begin{pmatrix}
    2 & 0 & 0 \\
    0 & 4 & 0 \\
    0 & 0 & 30
    \end{pmatrix}$
     &  - \\ 
    \hline
    
    \rownumber. & 4 & 1 & $M_{10}$ & 2 & $\begin{pmatrix}
    2 & 0 \\
    0 & 30
    \end{pmatrix}$
     & f & $\begin{pmatrix}
    2 & 0 & 0 \\
    0 & 4 & 0 \\
    0 & 0 & 30
    \end{pmatrix}$
     &  - \\ 
    \hline
    
    \rownumber. & 4 & 1 & $\mathbb{Z}_2^4:(S_3 \times S_3)$ & 2 & $\begin{pmatrix}
    6 & 0 \\
    0 & 24
    \end{pmatrix}$
     & f & $\begin{pmatrix}
    4 & 0 & 0 \\
    0 & 6 & 0 \\
    0 & 0 & 24
    \end{pmatrix}$
     &  - \\ 
    \hline
    
    \rownumber. & 6 & 1 & $A_7$ & 2 & $\begin{pmatrix}
    2 & 0 \\
    0 & 70
    \end{pmatrix}$
     & f & $\begin{pmatrix}
    2 & 1 & 0 \\
    1 & 2 & 0 \\
    0 & 0 & 70
    \end{pmatrix}$
     &  - \\ 
    \hline
    
    \rownumber. & 6 & 1 & $A_7$ & 2 & $\begin{pmatrix}
    8 & 2 \\
    2 & 18
    \end{pmatrix}$
     & f & $\begin{pmatrix}
    6 & 3 & 1 \\
    3 & 6 & 1 \\
    1 & 1 & 8
    \end{pmatrix}$
     &  - \\ 
    \hline
    
    \rownumber. & 6 & 1 & $\mathbb{Z}_2:A_6$ & 2 & $\begin{pmatrix}
    4 & 0 \\
    0 & 8
    \end{pmatrix}$
     &  -  & $\begin{pmatrix}
    4 & 0 & 0 \\
    0 & 6 & 0 \\
    0 & 0 & 8
    \end{pmatrix}$
     &  - \\ 
    \hline
    
    \rownumber. & 6 & 1 & $(\mathbb{Z}_3 \times A_5):\mathbb{Z}_2$ & 2 & $\begin{pmatrix}
    10 & 0 \\
    0 & 30
    \end{pmatrix}$
     & f & $\begin{pmatrix}
    4 & 1 & 0 \\
    1 & 4 & 0 \\
    0 & 0 & 30
    \end{pmatrix}$
     &  - \\ 
    \hline
    
    \rownumber. & 6 & 1 & $\mathbb{Z}_2^4:(S_3 \times S_3)$ & 2 & $\begin{pmatrix}
    4 & 0 \\
    0 & 24
    \end{pmatrix}$
     &  -  & $\begin{pmatrix}
    4 & 0 & 0 \\
    0 & 6 & 0 \\
    0 & 0 & 24
    \end{pmatrix}$
     &  - \\ 
    \hline
    
    \rownumber. & 6 & 1 & $3^{1+4} : 2.2^2$ & 2 & $\begin{pmatrix}
    6 & 0 \\
    0 & 6
    \end{pmatrix}$
     & f & $\begin{pmatrix}
    6 & 0 & 0 \\
    0 & 6 & 0 \\
    0 & 0 & 6
    \end{pmatrix}$
     &  - \\ 
    \hline
    
    \rownumber. & 6 & 1 & $3^4 : A_6$ & 2 & $\begin{pmatrix}
    6 & 0 \\
    0 & 18
    \end{pmatrix}$
     & f & $\begin{pmatrix}
    6 & 3 & 0 \\
    3 & 6 & 0 \\
    0 & 0 & 6
    \end{pmatrix}$
     &  - \\ 
    \hline
    
    \rownumber. & 6 & 2 & $L_2(11)$ & 3 & $\begin{pmatrix}
    22 & 11 \\
    11 & 22
    \end{pmatrix}$
     &  -  & $\begin{pmatrix}
    6 & 2 & 2 \\
    2 & 8 & -3 \\
    2 & -3 & 8
    \end{pmatrix}$
     &  \ref{6d2_L2(11)} \\ 
    \hline
    
    \rownumber. & 6 & 2 & $A_7$ & 2 & $\begin{pmatrix}
    2 & 1 \\
    1 & 18
    \end{pmatrix}$
     &  -  & $\begin{pmatrix}
    2 & 0 & 1 \\
    0 & 6 & 0 \\
    1 & 0 & 18
    \end{pmatrix}$
     &  \ref{6d2_A_7} \\ 
    \hline
    
    \rownumber. & 6 & 2 & $(\mathbb{Z}_3 \times A_5):\mathbb{Z}_2$ & 2 & $\begin{pmatrix}
    10 & 5 \\
    5 & 10
    \end{pmatrix}$
     &  -  & $\begin{pmatrix}
    6 & 0 & 0 \\
    0 & 10 & 5 \\
    0 & 5 & 10
    \end{pmatrix}$
     &  - \\ 
    \hline
    
    \rownumber. & 6 & 2 & $(\mathbb{Z}_3 \times A_5):\mathbb{Z}_2$ & 6 & $\begin{pmatrix}
    10 & 5 \\
    5 & 10
    \end{pmatrix}$
     &  -  & $\begin{pmatrix}
    6 & 0 & 0 \\
    0 & 10 & 5 \\
    0 & 5 & 10
    \end{pmatrix}$
     &  \ref{6d2_Z3A_5:Z2} \\ 
    \hline
    
    \rownumber. & 6 & 2 & $3^{1+4} : 2.2^2$ & 2 & $\begin{pmatrix}
    6 & 0 \\
    0 & 6
    \end{pmatrix}$
     &  -  & $\begin{pmatrix}
    6 & 0 & 0 \\
    0 & 6 & 0 \\
    0 & 0 & 6
    \end{pmatrix}$
     &  - \\ 
    \hline
    
    \rownumber. & 6 & 2 & $3^{1+4} : 2.2^2$ & 4 & $\begin{pmatrix}
    6 & 0 \\
    0 & 6
    \end{pmatrix}$
     &  -  & $\begin{pmatrix}
    6 & 0 & 0 \\
    0 & 6 & 0 \\
    0 & 0 & 6
    \end{pmatrix}$
     &   \ref{6d2_digits} \\ 
    \hline
    
    \rownumber. & 6 & 2 & $3^4 : A_6$ & 3 & $\begin{pmatrix}
    6 & 3 \\
    3 & 6
    \end{pmatrix}$
     &  -  & $\begin{pmatrix}
    6 & 3 & 0 \\
    3 & 6 & 0 \\
    0 & 0 & 6
    \end{pmatrix}$
     &  - \\ 
    \hline
    
    \rownumber. & 6 & 2 & $3^4 : A_6$ & 2 & $\begin{pmatrix}
    6 & 3 \\
    3 & 6
    \end{pmatrix}$
     &  -  & $\begin{pmatrix}
    6 & 3 & 0 \\
    3 & 6 & 0 \\
    0 & 0 & 6
    \end{pmatrix}$
     &  - \\ 
    \hline
    
    \rownumber. & 6 & 2 & $3^4 : A_6$ & 6 & $\begin{pmatrix}
    6 & 3 \\
    3 & 6
    \end{pmatrix}$
     &  -  & $\begin{pmatrix}
    6 & 3 & 0 \\
    3 & 6 & 0 \\
    0 & 0 & 6
    \end{pmatrix}$
     &  \ref{6d2_3^4:A_6} \\ 
    \hline
    
    \rownumber. & 8 & 1 & $\mathbb{Z}_2:A_6$ & 2 & $\begin{pmatrix}
    4 & 0 \\
    0 & 6
    \end{pmatrix}$
     &  -  & $\begin{pmatrix}
    4 & 0 & 0 \\
    0 & 6 & 0 \\
    0 & 0 & 8
    \end{pmatrix}$
     &  - \\ 
    \hline
    
    \rownumber. & 8 & 1 & $\mathbb{Z}_2^4:S_5$ & 2 & $\begin{pmatrix}
    4 & 0 \\
    0 & 10
    \end{pmatrix}$
     & f & $\begin{pmatrix}
    4 & 0 & 0 \\
    0 & 8 & 0 \\
    0 & 0 & 10
    \end{pmatrix}$
     &  - \\ 
    \hline
    
    \rownumber. & 10 & 1 & $\mathbb{Z}_2^4:S_5$ & 2 & $\begin{pmatrix}
    4 & 0 \\
    0 & 8
    \end{pmatrix}$
     &  -  & $\begin{pmatrix}
    4 & 0 & 0 \\
    0 & 8 & 0 \\
    0 & 0 & 10
    \end{pmatrix}$
     &  - \\ 
    \hline
    
    \rownumber. & 10 & 1 & $(\mathbb{Z}_3 \times A_5):\mathbb{Z}_2$ & 2 & $\begin{pmatrix}
    6 & 0 \\
    0 & 30
    \end{pmatrix}$
     & f & $\begin{pmatrix}
    4 & 1 & 0 \\
    1 & 4 & 0 \\
    0 & 0 & 30
    \end{pmatrix}$
     &  - \\ 
    \hline
    
    \rownumber. & 10 & 1 & $(\mathbb{Z}_3 \times A_5):\mathbb{Z}_2$ & 2 & $\begin{pmatrix}
    6 & 0 \\
    0 & 30
    \end{pmatrix}$
     & f & $\begin{pmatrix}
    6 & 0 & 0 \\
    0 & 10 & 5 \\
    0 & 5 & 10
    \end{pmatrix}$
     &  - \\ 
    \hline
    
    \rownumber. & 12 & 1 & $L_3(4)$ & 2 & $\begin{pmatrix}
    2 & 0 \\
    0 & 28
    \end{pmatrix}$
     & f & $\begin{pmatrix}
    2 & 0 & 0 \\
    0 & 10 & 4 \\
    0 & 4 & 10
    \end{pmatrix}$
     &  - \\ 
    \hline
    
    \rownumber. & 12 & 1 & $L_3(4)$ & 2 & $\begin{pmatrix}
    4 & 0 \\
    0 & 14
    \end{pmatrix}$
     & f & $\begin{pmatrix}
    4 & 2 & 0 \\
    2 & 4 & 0 \\
    0 & 0 & 14
    \end{pmatrix}$
     &  - \\ 
    \hline
    
    \rownumber. & 12 & 1 & $S_6$ & 2 & $\begin{pmatrix}
    4 & 0 \\
    0 & 30
    \end{pmatrix}$
     & f & $\begin{pmatrix}
    4 & 2 & 0 \\
    2 & 4 & 0 \\
    0 & 0 & 30
    \end{pmatrix}$
     &  - \\ 
    \hline
    
    \rownumber. & 12 & 1 & $M_{10}$ & 2 & $\begin{pmatrix}
    4 & 2 \\
    2 & 6
    \end{pmatrix}$
     & f & $\begin{pmatrix}
    4 & 2 & 0 \\
    2 & 6 & 0 \\
    0 & 0 & 12
    \end{pmatrix}$
     &  - \\ 
    \hline
    
    \rownumber. & 12 & 1 & $\mathbb{Z}_3^2:QD_{16}$ & 2 & $\begin{pmatrix}
    4 & 2 \\
    2 & 10
    \end{pmatrix}$
     &  -  & $\begin{pmatrix}
    4 & 2 & 0 \\
    2 & 10 & 0 \\
    0 & 0 & 12
    \end{pmatrix}$
     &  - \\ 
    \hline
    
    \rownumber. & 12 & 1 & $3^{1+4} : 2.2^2$ & 2 & $\begin{pmatrix}
    6 & 0 \\
    0 & 12
    \end{pmatrix}$
     & f & $\begin{pmatrix}
    6 & 0 & 0 \\
    0 & 6 & 0 \\
    0 & 0 & 6
    \end{pmatrix}$
     &  - \\ 
    \hline
    
    \rownumber. & 14 & 1 & $\mathbb{Z}_2 \times L_2(7)$ & 2 & $\begin{pmatrix}
    2 & 0 \\
    0 & 14
    \end{pmatrix}$
     &  -  & $\begin{pmatrix}
    2 & 0 & 0 \\
    0 & 14 & 0 \\
    0 & 0 & 14
    \end{pmatrix}$
     &  - \\ 
    \hline
    
    \rownumber. & 14 & 2 & $L_3(4)$ & 6 & $\begin{pmatrix}
    4 & 2 \\
    2 & 4
    \end{pmatrix}$
     &  -  & $\begin{pmatrix}
    4 & 2 & 0 \\
    2 & 4 & 0 \\
    0 & 0 & 14
    \end{pmatrix}$
     &  - \\ 
    \hline
    
    \rownumber. & 14 & 2 & $L_3(4)$ & 2 & $\begin{pmatrix}
    4 & 2 \\
    2 & 4
    \end{pmatrix}$
     &  -  & $\begin{pmatrix}
    4 & 2 & 0 \\
    2 & 4 & 0 \\
    0 & 0 & 14
    \end{pmatrix}$
     &  - \\ 
    \hline
    
    \rownumber. & 14 & 2 & $\mathbb{Z}_2 \times L_2(7)$ & 2 & $\begin{pmatrix}
    4 & 2 \\
    2 & 8
    \end{pmatrix}$
     &  -  & $\begin{pmatrix}
    4 & 2 & 0 \\
    2 & 8 & 0 \\
    0 & 0 & 14
    \end{pmatrix}$
     &  - \\ 
    \hline
    
    \rownumber. & 16 & 1 & $Q(\mathbb{Z}_3^2:\mathbb{Z}_2)$ & 2 & $\begin{pmatrix}
    8 & 4 \\
    4 & 14
    \end{pmatrix}$
     & f & $\begin{pmatrix}
    6 & 2 & 2 \\
    2 & 6 & -2 \\
    2 & -2 & 14
    \end{pmatrix}$
     &  - \\ 
    \hline
    
    \rownumber. & 18 & 1 & $3^4 : A_6$ & 2 & $\begin{pmatrix}
    6 & 0 \\
    0 & 6
    \end{pmatrix}$
     & f & $\begin{pmatrix}
    6 & 3 & 0 \\
    3 & 6 & 0 \\
    0 & 0 & 6
    \end{pmatrix}$
     &  - \\ 
    \hline
    
    \rownumber. & 22 & 1 & $L_2(11)$ & 2 & $\begin{pmatrix}
    2 & 0 \\
    0 & 22
    \end{pmatrix}$
     & f & $\begin{pmatrix}
    2 & 1 & 0 \\
    1 & 6 & 0 \\
    0 & 0 & 22
    \end{pmatrix}$
     &  - \\ 
    \hline
    
    \rownumber. & 22 & 1 & $L_2(11)$ & 2 & $\begin{pmatrix}
    6 & 2 \\
    2 & 8
    \end{pmatrix}$
     & f & $\begin{pmatrix}
    6 & 2 & 2 \\
    2 & 8 & -3 \\
    2 & -3 & 8
    \end{pmatrix}$
     &  - \\ 
    \hline
    
    \rownumber. & 22 & 2 & $L_2(11)$ & 2 & $\begin{pmatrix}
    2 & 1 \\
    1 & 6
    \end{pmatrix}$
     &  -  & $\begin{pmatrix}
    2 & 1 & 0 \\
    1 & 6 & 0 \\
    0 & 0 & 22
    \end{pmatrix}$
     &  \ref{22d2_L2(11)} \\ 
    \hline
    
    \rownumber. & 24 & 1 & $\mathbb{Z}_2:A_6$ & 2 & $\begin{pmatrix}
    2 & 0 \\
    0 & 4
    \end{pmatrix}$
     &  -  & $\begin{pmatrix}
    2 & 0 & 0 \\
    0 & 4 & 0 \\
    0 & 0 & 24
    \end{pmatrix}$
     &  - \\ 
    \hline
    
    \rownumber. & 24 & 1 & $\mathbb{Z}_2^4:(S_3 \times S_3)$ & 2 & $\begin{pmatrix}
    4 & 0 \\
    0 & 6
    \end{pmatrix}$
     &  -  & $\begin{pmatrix}
    4 & 0 & 0 \\
    0 & 6 & 0 \\
    0 & 0 & 24
    \end{pmatrix}$
     &  - \\ 
    \hline
    
    \rownumber. & 28 & 1 & $L_3(4)$ & 2 & $\begin{pmatrix}
    2 & 0 \\
    0 & 12
    \end{pmatrix}$
     & f & $\begin{pmatrix}
    2 & 0 & 0 \\
    0 & 10 & 4 \\
    0 & 4 & 10
    \end{pmatrix}$
     &  - \\ 
    \hline
    
    \rownumber. & 28 & 1 & $\mathbb{Z}_2 \times L_2(7)$ & 2 & $\begin{pmatrix}
    2 & 0 \\
    0 & 28
    \end{pmatrix}$
     & f & $\begin{pmatrix}
    2 & 0 & 0 \\
    0 & 14 & 0 \\
    0 & 0 & 14
    \end{pmatrix}$
     &  - \\ 
    \hline
    
    \rownumber. & 28 & 1 & $\mathbb{Z}_2 \times L_2(7)$ & 2 & $\begin{pmatrix}
    4 & 0 \\
    0 & 14
    \end{pmatrix}$
     & f & $\begin{pmatrix}
    4 & 2 & 0 \\
    2 & 8 & 0 \\
    0 & 0 & 14
    \end{pmatrix}$
     &  - \\ 
    \hline
    
    \rownumber. & 30 & 1 & $M_{10}$ & 2 & $\begin{pmatrix}
    2 & 0 \\
    0 & 4
    \end{pmatrix}$
     & f & $\begin{pmatrix}
    2 & 0 & 0 \\
    0 & 4 & 0 \\
    0 & 0 & 30
    \end{pmatrix}$
     &  - \\ 
    \hline
    
    \rownumber. & 30 & 1 & $(\mathbb{Z}_3 \times A_5):\mathbb{Z}_2$ & 2 & $\begin{pmatrix}
    6 & 0 \\
    0 & 10
    \end{pmatrix}$
     & f & $\begin{pmatrix}
    6 & 0 & 0 \\
    0 & 10 & 5 \\
    0 & 5 & 10
    \end{pmatrix}$
     &  - \\ 
    \hline
    
    \rownumber. & 30 & 2 & $S_6$ & 2 & $\begin{pmatrix}
    4 & 2 \\
    2 & 4
    \end{pmatrix}$
     &  -  & $\begin{pmatrix}
    4 & 2 & 0 \\
    2 & 4 & 0 \\
    0 & 0 & 30
    \end{pmatrix}$
     &  - \\ 
    \hline
    
    \rownumber. & 30 & 2 & $M_{10}$ & 2 & $\begin{pmatrix}
    2 & 0 \\
    0 & 4
    \end{pmatrix}$
     &  -  & $\begin{pmatrix}
    2 & 0 & 0 \\
    0 & 4 & 0 \\
    0 & 0 & 30
    \end{pmatrix}$
     &  - \\ 
    \hline
    
    \rownumber. & 30 & 2 & $(\mathbb{Z}_3 \times A_5):\mathbb{Z}_2$ & 2 & $\begin{pmatrix}
    4 & 1 \\
    1 & 4
    \end{pmatrix}$
     &  -  & $\begin{pmatrix}
    4 & 1 & 0 \\
    1 & 4 & 0 \\
    0 & 0 & 30
    \end{pmatrix}$
     &  - \\ 
    \hline
    
    \rownumber. & 40 & 1 & $\mathbb{Z}_2^4:S_5$ & 2 & $\begin{pmatrix}
    2 & 0 \\
    0 & 4
    \end{pmatrix}$
     &  -  & $\begin{pmatrix}
    2 & 0 & 0 \\
    0 & 4 & 0 \\
    0 & 0 & 40
    \end{pmatrix}$
     &  - \\ 
    \hline
    
    \rownumber. & 42 & 1 & $A_7$ & 2 & $\begin{pmatrix}
    4 & 2 \\
    2 & 6
    \end{pmatrix}$
     & f & $\begin{pmatrix}
    4 & 2 & 1 \\
    2 & 6 & 3 \\
    1 & 3 & 12
    \end{pmatrix}$
     &  - \\ 
    \hline
    
    \rownumber. & 70 & 1 & $A_7$ & 2 & $\begin{pmatrix}
    2 & 0 \\
    0 & 6
    \end{pmatrix}$
     & f & $\begin{pmatrix}
    2 & 0 & 1 \\
    0 & 6 & 0 \\
    1 & 0 & 18
    \end{pmatrix}$
     &  - \\ 
    \hline
    
    \rownumber. & 70 & 2 & $A_7$ & 6 & $\begin{pmatrix}
    2 & 1 \\
    1 & 2
    \end{pmatrix}$
     &  -  & $\begin{pmatrix}
    2 & 1 & 0 \\
    1 & 2 & 0 \\
    0 & 0 & 70
    \end{pmatrix}$
     &  - \\ 
    \hline
    
    \rownumber. & 70 & 2 & $A_7$ & 2 & $\begin{pmatrix}
    2 & 1 \\
    1 & 2
    \end{pmatrix}$
     &  -  & $\begin{pmatrix}
    2 & 1 & 0 \\
    1 & 2 & 0 \\
    0 & 0 & 70
    \end{pmatrix}$
     &  - \\ 
    \hline
    
    \caption{List of group actions}
    \label{table:1}
    \end{longtable}
\end{center}

\normalsize

\section{Known examples of very symmetric fourfolds}\label{examples}

\subsection{Fano varieties of lines}

Let $V \subset \mathbb{P}^5$ be a smooth cubic fourfold, let $\Grass(2, 6)$ be the Grassmannian of planes in a 6-space (isomorphically, projective lines in the projective 5-space). Let us consider the following subvariety:
\begin{align*}
    F(V) = \{L \in \Grass(2, 6)| L \subset V\},
\end{align*}
it is called the \textit{Fano variety (or scheme) of lines} of $X$. It is naturally embedded in $\mathbb{P}^{14}$ (via its embedding in $\Grass(2, 6)$). It turns out to be a \KTST{} for generic $V$. The Pl{\"u}cker line bundle on $\Grass(2, 6)$ restricted to $F(V)$ induces a polarization of \BB{} degree 6 on the lattice $\Homology^2(F(V), \mathbb{Z})$ of divisibility 2, let us denote it by $h$. 
In \cite{BD}, it was shown that a generic polarized manifold of K3$^{[2]}$-type with polarization of \BB{} degree 6 and divisibility 2 can be described this way. 

In this section we shall construct special cubic fourfolds $V$ such that $F(V)$ is a very symmetric \KTST{}. We find in this way geometric constructions of the examples for all the entries (rows) with $h$ of \BB{} degree $6$ and divisibility $2$ from the table (cf.~\cite[\S 8]{HM} \cite{Fu}, \cite[Ch.~4]{Mongardi}). 
In order to prove that a given $F(V)$ admits the required group of automorphisms we 
use the following observation from \cite[Thm.~0.1]{Fu}, and easily conclude the following (via case-by-case analysis). 

\begin{fact*}
A projective isomorphism $f$ of $\mathbb{P}^5$ of order 2, 3 (and therefore also 6), or 4 leaving a cubic fourfold $V$ invariant induces a symplectic automorphism on a Fano scheme $F(V)$ if and only if the matrix representing $f$, normalized so it has an eigenvalue 1, has determinant 1.
\end{fact*}

For the duration of the section, $x_0,x_1,\ldots, x_5$ will be coordinates in $\mathbb{P}^5$.
    \begin{example}\label{6d2_3^4:A_6}
        Let $V$ be the \textit{Fermat's cubic} defined by the equation 
        \begin{align*}
            x_0^3 + x_1^3 + x_2^3 + x_3^3 + x_4^3 + x_5^3 = 0.
        \end{align*}
         One easily sees the action of $3^5 \colon S_6$ on this variety which consists of permutations of variables, and multiplications of variables by roots of unity of degree 3. The action induces an action by the same group on $F(V)$. 
        By the above fact, one can check that the subgroup $3^4 \colon A_6$ acts symplectically on $F(V)$. By comparing with the table, we conclude that
        \begin{align*}
             \T_{F(V)} \cong \begin{pmatrix} 6 & 3 \\ 3 & 6 \end{pmatrix}.
        \end{align*}
        \begin{proposition*}
            The above is cardinality-wise the biggest possible group of polarized automorphisms acting on a \KTST{}, said cardinality being 174960.
        \end{proposition*}
        \begin{proof}
              This is because $3^4 \colon A_6$ is the biggest possible group that can act by symplectic automorphisms on a \KTST{} $X$ and $3^5 \colon S_6$ is 6 times bigger than it. By Lemma \ref{trace_lemma} a group of automorphisms $G$ on such $X$ extending a group of symplectic automorphisms $\widetilde{G}$ can be at most 6 times bigger than $\widetilde{G}$ as long as $\rk\Homology^2(X, \mathbb{Z})_{\widetilde{G}} = 20$. In general, one has $|G|/|\widetilde{G}| \le 66$ (cf. \cite[Corollary 7.1.5]{Mongardi}). But the biggest possible $\widetilde{G}$ for which the rank of the coinvariant lattice can drop below 20 has cardinality 972 (cf. \cite[Table 12]{HM}), and obviously $66 \cdot 972 < 174960$.
         \end{proof}
         It makes this manifold the most symmetric \KTST{} in a sense. 
    \end{example}

    \begin{example}\label{6d2_L2(11)}
         Let $V$ be the cubic fourfold defined by the equation  
        \begin{align*}
            x_0^3 + x_1^2x_2 + x_2^2x_3 + x_3^2x_4 + x_4^2x_5 + x_5^2x_1 = 0.
        \end{align*}
        Now let us consider the cubic threefold $\widetilde{V}$ given by the equation 
        \begin{align*}
                y_0^2y_1 + y_1^2y_2 + y_2^2y_3 + y_3^2y_4 + y_4^2y_0 = 0,
        \end{align*}
        where $y_0, \ldots, y_5$ are the coordinates in $\mathbb{P}^4$. $\widetilde{V}$ has the automorphism group $L_2(11)$ (as proven in \cite{Adler}). As noted in \cite{Mongardi}, $V$ is a ramified $3:1$ cover of $\widetilde{V}$, from which we get an obvious action of $\mathbb{Z}/3 \times L_2(11)$ on $V$ and therefore a polarized action on $F(V)$. Since $L_2(11)$ is simple (as a projective linear group of a field with more than 3 elements) by Lemma \ref{simple_lemma}, $L_2(11) \subset \mathbb{Z}/3 \times L_2(11)$ is symplectic. So from Table \ref{table:1}:
        \begin{align*}
            \T_{F(V)} = \begin{pmatrix} 22 & 11 \\ 11 & 22 \end{pmatrix}.
        \end{align*}
    \end{example}
    
    \begin{example}\label{6d2_A_7}
         Let $V$ be the cubic defined by the equation
        \begin{align*}
            x_0^3 + x_1^3 + x_2^3 + x_3^3 + x_4^3 + x_5^3 - (x_0 + x_1 + x_2 + x_3 + x_4 + x_5)^3 = 0.
        \end{align*}
        Any two of the seven summands can be transformed into each other without altering the original equation, giving us an action of $S_7$\footnote{\cite{Mongardi} claims the existence of an action of $S_7.3$ with the $\mathbb{Z}/3$ action construed analogously to the previous example, considering $V$ as a ramified covering. However such morphisms will not be automorphisms.}  which induces a polarized action on ($F(V), h)$. Noting that $A_7$ is simple, we can again deduce from (\ref{the_exact_sequence}) that it acts symplectically on $F(V)$. This leads us to
        \begin{align*}
            \T_{F(V)} \cong \begin{pmatrix} 2 & 1 \\ 1 & 18 \end{pmatrix}.
        \end{align*}
    \end{example}
    
    \begin{example}\label{6d2_Z3A_5:Z2}
        Let $V$ be the cubic defined by the equation
        \begin{align*}
            x_0^2x_1 + x_1^2x_2 + x_2^2x_3 + x_3^2x_1 + x_4^3 + x_5^3 = 0.
        \end{align*}
         In\cite{Mongardi}, the cubic  $C \subset \mathbb{P}^3$ defined by the equation
        \begin{align*}
            x_0^2x_1 + x_1^2x_2 + x_2^2x_3 + x_3^2x_1 = 0,
        \end{align*}
        is considered, by use known classifications, the author concludes that $C$ is isomorphic to the surface in $\mathbb{P}^4$ given by the equations
        \begin{align*}
            \begin{cases}
                y_0^3 + y_1^3 + y_2^3 + y_3^3 + y_4^3 = 0, \\
                y_0 + y_1 + y_2 + y_3 + y_4 = 0.
            \end{cases}
        \end{align*}
        The last variety has automorphism group isomorphic to $S_5$. Obviously they all induce automorphisms of $V$. However, only the subgroup $A_5$ will induce symplectic automorphisms on $F(V)$, but remaining members of $S_5$ will uniquely correspond to members of $Aut_s(F(V))$ by first composing them with the transposition $\phi \colon (x_0,\ldots, x_4, x_5) \mapsto (x_0,\ldots, x_5, x_4)$ on $V \subset \mathbb{P}^5$. There is also an automorphism $\text{diag}(\zeta_{15}, \zeta_{15}^{13}, \zeta_{15}^4, \zeta_{15}^7, \zeta_{15}^5, 1)$ where $\text{diag}(\lambda_1, \ldots, \lambda_n)$ denotes the diagonal $n \times n$ matrix with the given values on the diagonal (in order) and $\zeta_{15}$ is a fixed primitive root of unity of order 15, it will induce a symplectic automorphism of $F(V)$ too, its order when restricted to $V$ is 3, it commutes with previously discussed $A_5$ subgroup. Together the aforementioned morphisms generate a group isomorphic to $((\mathbb{Z}/3) \times A_5) \colon \mathbb{Z}/2$. We note however that there are also other, nonsymplectic, polarized automorphisms of $F(V)$. We will get all the finite polarized automorphisms if we extend the group inducing symplectic automorphisms, the morphism $\phi$, and the order three morphism $\text{diag}(1, 1, 1, 1, 1, \zeta_{15}^5)$, so in $PGL(6, \mathbb{C}$) we will get a group six times larger than $((\mathbb{Z}/3) \times A_5) \colon \mathbb{Z}/2$ and normalizing it, so we must have the following transcendental lattice
        \begin{align*}
            \T_{F(V)} \cong \begin{pmatrix} 10 & 5 \\ 5 & 10\end{pmatrix}.
        \end{align*}
    \end{example}
    \begin{example}\label{6d2_digits}
        One of the Fano schemes originally constructed in \cite{HM} is $F(V)$ for $V$ the cubic defined by the equation
        \begin{align*}
            x_0^3 + x_1^3 + x_2^3 + x_3^3 + x_4^3 + x_5^3 + \lambda (x_0 x_1 x_2 + x_3 x_4 x_5),
        \end{align*}
         where $\lambda = 3(i - 2e^\frac{\pi i}{6} -1)$. Aside from the obvious permutations and the variable-wise multiplications by the third degree roots of unity, the authors also find an additional linear isomorphism\footnote{Given by the matrix $\frac{1}{\sqrt{3}}\left[\begin{matrix}
        \omega & \omega^2 & 1 & & & \\
        1 & 1 & 1 & & & \\
        \omega^2 & \omega & 1 & & & \\
        & & & \omega^2 & \omega  & 1 \\
        & & & \omega^2 & \omega^2  & \omega^2 \\
        & & & \omega^2 & 1 & \omega 
        \end{matrix}\right]$, for $\omega = 2e^\frac{2 \pi i}{3}$.} leaving the equation invariant. The group generated by their images in $\text{PSL}(6, \mathbb{C})$ is isomorphic to a group of the form $3^{4 + 1} \colon 2 . 2^2$. This group will act symplectically on the Fano scheme, and a calculation shows that images of the generators in $\text{PGL}(6, \mathbb{C})$ will give us the group extended by 4 which according to means $F(V)$ has an action of the maximal possible overgroup of the studied group, and 
        \begin{align*}
            \T_{F(V)} \cong \begin{pmatrix} 6 & 0 \\ 0 & 6 \end{pmatrix}.
        \end{align*}
    \end{example}
    
\begin{remark*}
    In \cite[Therem 1.8]{LazaZheng}, the authors describe all the smooth cubic fourfolds which admit actions of some of the 15 maximal groups from \cite{HM}. In particular, the above 5 examples are the only such cubics which admit actions strictly containing any such group. That means other examples from Table \ref{table:1} of polarized \KTSTs{} of \BB{} degree 6, divisibility 2 cannot be described as varieties of lines of such manifolds. Note that not all \hk{} fourfolds admitting polarization like this are related to smooth cubics. The cases where it does not happen are described in detail in \cite{Looijenga}. 
\end{remark*}

\subsection{Debarre-Voisin varieties}\label{22_L2(11)}
Let us fix a ten dimensional complex vector space $V_{10} \cong \mathbb{C}^{10}$ and a three-form $\sigma \in (\bigwedge\nolimits^3 V_{10})^{*} \cong \bigwedge\nolimits^3 V_{10}^{*}$. We define a subvariety of the Grassmannian $\Grass(6, V_{10})$:
\begin{align} \label{DV_Variety}
    \text{DV}(\sigma) = \{V_{6} \in \Grass(6, V_{10}) \colon \sigma_{|\bigwedge^3 V_6} \equiv 0 \}.
\end{align}
In \cite{DV}, it is proven that it is a \KTST{} for which there exists a polarization of \BB{} 22 of divisibility 2 (and a generic such polarized manifold is of the form (\ref{DV_Variety})). 
\par

\begin{example}\label{22d2_L2(11)}
        Let us consider $G$ a subgroup of the general linear group of $V_{10}$. Its action on $V_{10}$ induces an action on the Grassmannian $\Grass(6, V_{10})$. If for all $g \in G$, and $u, v, w \in V_{10}$, there exists a nonzero constant $\lambda_{g,u,v,w}$, such that 
\begin{align*}
    \sigma(g(u) \wedge g(v) \wedge g(w)) = \lambda_{g,u,v,w} \sigma(u \wedge v \wedge w), 
\end{align*}
then by simple linear algebra, $G$ restricts to an action on $\text{DV}(\sigma)$. 
Consider now the representation\footnote{It is one of the two irreducible representations of $L_2(11)$ in dimension 10, the other one yields no results however.} of the group $\widetilde{G} = L_2(11)$ on $V_{10}$ given by the following two generators:

\begin{align*}
\begin{split}
     g_1  &= \left[ \begin{matrix} 0 & 1 & 0 & 0 & 0 & 0 & 0 & 0 & 0 & 0 \\
    1 & 0 & 0 & 0 & 0 & 0 & 0 & 0 & 0 & 0 \\
    0 & 0 & 0 & 0 & 1 & 0 & 0 & 0 & 0 & 0 \\
    0 & 0 & 0 & 0 & 0 & 1 & 0 & 0 & 0 & 0 \\
    0 & 0 & 1 & 0 & 0 & 0 & 0 & 0 & 0 & 0 \\
    0 & 0 & 0 & 1 & 0 & 0 & 0 & 0 & 0 & 0 \\
    0 & 0 & 0 & 0 & 0 & 0 & 0 & 0 & 1 & 0 \\
    -1 & -1 & -1 & -1 & -1 & -1 & 1 & -1 & 1 & 0 \\
    0 & 0 & 0 & 0 & 0 & 0 & 1 & 0 & 0 & 0 \\
    0 & 0 & 0 & 1 & 0 & 1 & -1 & 0 & -1 & -1    \end{matrix} \right],\\
    g_2  &= \left[ \begin{matrix} 0 & 0 & 1 & 0 & 0 & 0 & 0 & 0 & 0 & 0 \\
    0 & 0 & 0 & 1 & 0 & 0 & 0 & 0 & 0 & 0 \\
    -1 & 0 & -1 & 0 & 0 & 0 & 0 & 0 & 0 & 0 \\
    0 & 0 & 0 & 0 & 0 & 0 & 1 & 0 & 0 & 0 \\
    0 & 1 & 0 & -1 & 1 & 0 & 0 & 0 & 0 & 0 \\
    0 & 0 & 0 & 0 & 0 & 0 & 0 & 1 & 0 & 0 \\
    0 & 1 & 0 & 0 & 0 & 0 & 0 & 0 & 0 & 0 \\
    0 & 0 & 0 & 0 & 0 & 0 & 0 & 0 & 0 & 1 \\
    -1 & 0 & 0 & -1 & 0 & -1 & 1 & 0 & 1 & 1 \\
    0 & 0 & 0 & 0 & 0 & 1 & 0 & 0 & 0 & 0 \end{matrix} \right].  
\end{split}
\end{align*}

After computing the induced mappings in the 120-dimensional space $\bigwedge^3V_{10}^*$, one finds that there are two eigenspaces for the respective induced actions of $g_1$ and $g_2$ which intersect non-trivially, the intersection is a 1-dimensional space. A form $\sigma_0$ from this space gives rise to $X = \text{DV}(\sigma_0)$, the manifold sought after. It has the action of $L_2(11)$ which must be symplectic because the group is simple. Now the invariant lattice for $X$ must be
\begin{align*}
    \T_X \cong \begin{pmatrix} 2 & 1 \\ 1 & 6 \end{pmatrix}.
\end{align*}
As mentioned in the previous section we know we can extend actions in the $L_2(11)$ case, so there exist a nonsymplectic involution for this manifold.
\par
That example has been independently considered in \cite{Song} where much more thorough study of the manifold follows. Note that while we can prove the existence of an additional nonsymplectic automorphism, we do not know its geometric description. 
\end{example}

\subsection{Hilbert squares of quartics in \texorpdfstring{$\mathbb{P}^3$}{\space}}\label{squares}
Before giving the two examples, let us start with the following lemma.
\begin{lemma*}
    Let $(S,h)$ be a polarized K3 surface, where $h$ is a class of a hyperplane section of $S$ and $h^2 = 4$, and $S$ does not contain a line. Then there exists $(X, h - \xi)$ a birational model of $(S^{[2]}, h - \xi)$ in which (the image of) $h - \xi$ is ample.
\end{lemma*}
\begin{proof}
    The ample cone is one of the (open) chambers inside the interior of the movable cone in $\NS_X \otimes \mathbb{R}$ cut by the -10 classes of divisibility 2. The movable cone itself is cut by -2 classes inside the positive cone (\cite{D}, section 3.7). Birational isomorphisms \KTST{} type manifolds induce isometries of the second cohomology lattices which transform the ample cone into one of the remaining chambers (and for each chamber there exists an appropriate birational isomorphism). The line bundle associated with $h$ is necessary nef (hence it lies in the closure of the ample cone); furthermore it lies on the boundary of the boundary of the movable cone as the -2-class $\xi$ lies in $\NS_X$. So $h$ is on the boundary of both the movable and ample cones; the movable cone is in the direction of $-\xi$. So for $h - \xi$ to lie inside one of the chambers, it is both necessary and sufficient for it to not be orthogonal to any -10-classes of divisibility 2 (so it does not lie on the boundary of a chamber) and for there to be no vector orthogonal to a -2-class on the segment from $h$ to $h - \xi$ with the exception of $h$ for which $\xi$ (and $-\xi$) must be the only such vector (so that the segment is contained in the movable cone) .
    
    Assume that there exists $y$, a -10-class of divisibility 2 perpendicular to $h - \xi$. For $y$ to be a primitive class of divisibility 2, it must be of the form $2kx + (2l+1)\xi$ for some nonzero integers $l$, $k$, and $x$ a primitive element from the K3 lattice\footnote{I.e.\ the unimodular lattice $\Homology^2(S, \mathbb{Z} \equiv L_{\text{K3}^{[2]}} = E_{8}(-1)^{\oplus 2} \oplus U^{\oplus 3}$ which naturally embeds in $\Homology^2(S^{[2]}, \mathbb{Z})$.}. Writing $y^2 = -10$ and $h \cdot y = 0$, we arrive at
    \begin{align*}
        \begin{cases}
            x^2 = 2\frac{l^2 + l - 1}{k^2}, \\
            h \cdot x = \frac{-2l - 1}{k}.
        \end{cases}
    \end{align*}
    Notice that the lattice generated by $x$ and $h$ must be embedded in the Picard lattice of $S$, so because $h^2 > 0$, then $h^2 \cdot x^2 < (h \cdot x)^2$ by the Hodge index theorem. Together with the above system of equations, it gives us the following possible Gram matrices for the lattice (up to the change of sign of $x$):
    
    \begin{align} 
        \label{line_in_a_quartic}
        \begin{pmatrix}-2 & 1 \\ 1 & 4\end{pmatrix},
    \end{align}
    \begin{align}
        \label{bad_sublattice}
        \begin{pmatrix}2 & 3 \\ 3 & 4\end{pmatrix}. 
    \end{align}
    We notice that the lattice (\ref{bad_sublattice}) contains the lattice (\ref{line_in_a_quartic}) as a sublattice with the basis $3x - 2h$ and $h$.
    
    Now assume there exists a -2-class perpendicular to some vector $h_t = h-t \xi$ where $0 < t \le 1$ is a real number (i.e.\ there's a wall separating $h$ and $h - \xi$). Write that -2-class as $y = k x + l \xi$ where $x$ is a primitive vector from $\Pic(S)$ and $l, \, k$ are integers. Since $k$ must be nonzero for $t \neq 0$, $l$ must be nonzero\footnote{Otherwise there would be $x^2 = -2$ and $h \cdot x = 0$, but $h$ is ample on the K3 surface $S$, so it must have nonzero intersection with every -2-class.}, from the equalities $y^2 = -2$ and $h_t \cdot y = 0$, we get
    \begin{align*}
        \begin{cases}
            x^2 = 2\frac{l^2 - 1}{k^2}, \\
            h \cdot x = 2t\frac{l}{k}.
        \end{cases}
    \end{align*}
    Since the sublattice generated by $x$ and $h$ must again be embedded in the Picard group, so again by the Hodge index theorem $h^2 \cdot x^2 < (h \cdot x)^2$, we arrive at $t^2 > 2 \frac{l^2-1}{l^2}$ which given the fact $t^2$ is at most 1, means $l^2 = 1$, and because $2t\frac{l}{k}$ is an integer, $t = \frac{1}{2}$ and $k^2 = 1$. So we again arrive on the lattice generated by $x$ and $h$ having a Gram matrix of the form (\ref{line_in_a_quartic}).
    
    In the case $t = 0$, let us show that $\xi$ and $-\xi$ are the only $-2$-classes perpendicular to $h_t = h$. Put $y = k x + l \xi$ as before and assume $k \neq 0$. We have
    \begin{align*}
        \begin{cases}
            x^2 = 2\frac{l^2 - 1}{k^2}, \\
            h \cdot x = 0.
        \end{cases}
    \end{align*}
    By the Hodge index theorem, $4 x^2 = h^2 \cdot x^2 < 0$. So $l^2 = 0$. But then $k x$ is a $-2$-class on $S$, so because $h$ is ample, $(kx)\cdot h \neq 0$. A contradiction.
    
    To end the proof, assume that there exists a class $x$ such that $x^2 = -2$ and $x \cdot h = 1$. Because $h$ is ample and $x \cdot h > 1$, $x$ is effective. So because $x^2 = -2$, there exists a curve $\Gamma \subset S$ of  class $x$. Now $h$ is a class of an intersection of $S$ with a hyperplane, so since $x \cdot h = 1$, $\Gamma$ has intersection number 1 with a hyperplane. So it is a line. But there are no lines inside $S$. A contradiction.

\end{proof}
\begin{remark*}
In fact in the Lemma above the polarisation $h-\xi$ is of degree $2$ and induces a $6:1$ map to a quadric in $\mathbb{P}^5$. The map associates to a length $2$ sub-scheme of the quartic the line spanned by it.
\end{remark*}
\begin{example}\label{2_2L2(7)}
        Consider the surface $S \subset \mathbb{P}^3$ given by the equation
    \begin{align*}
        x_0^3x_1 + x_1^3x_2 + x_2^3x_0 + x_3^4 = 0.
    \end{align*}
    It turns out to be the K3 surface with the transcendental lattice 
    \begin{align*}
        \T_S \cong \begin{pmatrix}14 & 0 \\ 0 & 14\end{pmatrix}.
    \end{align*}
    It admits a symplectic action of a group $H_0 = L_2(7)$ (effectively acting on the first three summands), and a nonsymplectic isometry of order four $\psi$ (last summand) that commutes with the action of $H_0$, $H = \langle H_0, \psi \rangle$ leaves an ample class $\widetilde{h}$ of square 4 invariant.
    One notices that the curve given by the equation $x_3 = 0$ is a hyperplane section of $S$ invariant under the action of $H$. So it is of class $\widetilde{h}$. Now by the lemma above, there exist $X$ a birational model of $S^{[2]}$ with $h = \widetilde{h}-\xi$ ample where via identification we write
    \begin{align*}
        L_X = L_S \oplus \langle \xi \rangle,
    \end{align*}
    for $L_S = \Homology^2(S, \mathbb{Z})$, $L_X = \Homology^2(X, \mathbb{Z})$ and $\xi^2 = -2$. Then $X$ inherits the $H$ action from $S$ with $H_0$ still being symplectic. We shall identify action of $H$ on both $S$ and $X$. Notice that the action of $H$ still preserves the vector $h$ and that $h^2 = 2$; also $h \cdot \T_X = h \cdot \T_S = 0$. Define the reflection $\rho$  by $h$ as
    \begin{align*}
        \rho \colon L_X \ni v \mapsto -v + (h \cdot v) h \in L_X.
    \end{align*}
    Note that $\rho$ is an isometry of $L_X$, and it also commutes with the action of $H$ on $X$. It acts as minus identity on $\T_X$, but so does $\psi$. Therefore the composition $\phi = \rho \circ \psi$ is a symplectic mapping not in $H_0$ and commuting with its elements. So we have a holomorphic action of $G = \langle H, \rho \rangle \cong \mathbb{Z}/2 \times \mathbb{Z}/4 \times L_2(7)$ with a subgroup of symplectic mappings $\widetilde{G} = \langle H, \psi \rangle \cong \mathbb{Z}/4 \times L_2(7)$. 
    
\end{example} 

\begin{example}\label{2_Z5_S5}
    Consider the surface $S \subset \mathbb{P}^3$ given by the equation
    \begin{align*}
        \sum_i x_i^4 - \sum_{i,j} x_i^2x_j^2 = 0.
    \end{align*}
    It is studied extensively in \cite{BS}, Section 4. It turns out to be the K3 surface with the transcendental lattice 
    \begin{align*}
        \T_S \cong \begin{pmatrix}4 & 0 \\ 0 & 40\end{pmatrix}.
    \end{align*}
    It admits a symplectic action of a group $H_0 = M_{20}$ and a nonsymplectic involution $\psi$ such that $|H| = 1920$ where $H = \langle H_0, \psi \rangle$ and $H$ leaves a hyperplane section class $\widetilde{h}$ of square 4 invariant. As $S$ again contains no lines, we find $X$ a birational model of $S^{[2]}$ which inherits the action of $H$ with an additional anti-symplectic involution $\rho$ commuting with $H$. So $X$ has a symplectic action of $\widetilde{G} = \langle H, \rho \circ \phi \rangle$ of order 1920 which by the classification \cite{HM} (Table 12) is the group $\mathbb{Z}_2^4:S_5$ that appears in our classification. 
\end{example}

\subsection{\DEPWs}
\begin{example}\label{EPW_L2(11)}
  In \cite{Mongardi} (Example 4.5.2), a so called \dEPW\ $X$ (which is \KTST{}, introduced in \cite{o2006irreducible}) with a symplectic action of $L_2(11)$ and an additional antisymplectic involution is constructed. Then we have
     \begin{align*}
         \T_{X} \cong \begin{pmatrix}
         22 & 0 \\ 0 & 22
         \end{pmatrix},
     \end{align*}
     and the action fixes a polarization $h$ with $h^2 = 2$.
\end{example}

\begin{example}\label{EPW_A7}
    In \cite{BW}, a joint work with Simone Billi, we construct two \dEPWs{} with an action of $\mathbb{Z}/2 \times A_7$. We prove that they are non-isomorphic as polarized manifolds, but we cannot say which correspond to which entry in the Table \ref{table:1} (or even whether they correspond to different entries).
\end{example}

\appendix

\section{Codes}\label{codes}
This subsection contains the codes\footnote{Sometimes abbreviated for conciseness; the full version can be found in the auxiliary files attached on arxiv.} for the computer aided computations done for the purpose of this paper as well as rationale for them.

\subsection{Gluing isometries}\label{main_code}

To recall the situation, we have a \fancyname{} isometry\footnote{See Notation \ref{good_notation}} $f$ on $L^{\widetilde{G}}$. We want to check whether it  extends to an isometry of the entire lattice $L$. We do so by checking the conditions from Lemma \ref{extension_condition} using the anti-embedding $\gamma \colon D_{L_{\widetilde{G}}} \rightarrow D_{L^{\widetilde{G}}}$ which exists by Lemma \ref{monomorphism}.

We note that according to \cite[Theorem 2.2]{HM} the pair $(L_{\widetilde{G}}(-1), \widetilde{G})$ is isometric\footnote{I.e.\ there exists an isometry $\phi \colon L_{\widetilde{G}}(-1) \rightarrow \Lambda_{\widetilde{G}'}$ and a group isomorphism $\iota \colon \widetilde{G} \rightarrow H$ such that $\phi \circ g \circ \phi^{-1} = \iota(g)$ for any $g \in \widetilde{G}$.} to $(\Lambda_{\widetilde{G}'}, \widetilde{G}')$ where $\Lambda$ is the Leech lattice\footnote{The unique positive-definite, unimodular, even lattice of rank 24 without vector $v$ such that $v^2 = 2$.} and $\widetilde{G}'$ is a group isomorphic to $\widetilde{G}$. That isometry induces an isometry of discriminant groups $D_{L_{\widetilde{G}}(-1)}$ and $D_{\Lambda_{\widetilde{G}}}$ which in turn gives an anti-isometry of $D_{L_{\widetilde{G}}}$ and $D_{\Lambda_{\widetilde{G}'}}$. Combined with the above paragraph, we may consider $\widetilde{\gamma} \colon D_{\Lambda_{\widetilde{G}}} \rightarrow D_{L^{\widetilde{G}}}$, an isometric embedding.

We note that not any embedding will work, i.e.\ not all such embedding will "glue" the two lattices into the lattice $L_{\text{K3}^{[2]}}$. However, based on \cite[Chapter 5, Theorem 13(b)]{sphere_pack}, $L_{\text{K3}^{[2]}}$ is the unique even lattice of signature (3, 20) such that its discriminant group has order 2 and the value of the quadratic form on the nontrivial element of the group is $\frac{3}{2} \modulo 2$. The lattice "glued" from $L^{\tilde{G}}$ and $L_{\tilde{G}}$ by any embedding $\gamma \colon D_{L_{\tilde{G}}} \rightarrow D_{L^{\tilde{G}}}$ will necessarily satisfy all conditions but the last one (up to change of the sign), so it is the only one we need to check.

So we begin by defining \texttt{Gram} to be the Gram matrix of the Leech lattice and initializing three lists: \texttt{group\_names}, \texttt{invariant\_lattices}, and\\ \texttt{coinvariant\_lattices} as empty. 
Below, we describe the code for constructing entries of the three lists based on an example of $\widetilde{G} = L_2(11)$.

\begin{lstlisting}[language=MAGMA]

//L_2(11)
lattices1 := []; 
Append(~lattices1, LatticeWithGram(
    Matrix([[2,1,0],[1,6,0],[0,0,22]]))
);
Append(~lattices1, LatticeWithGram(
    Matrix([[6,2,2],[2,8,-3],[2,-3,8]]))
);
//generators for G1 skipped for conciseness
G1 := MatrixGroup<24, IntegerRing() | /* ... */ >;
//Leech lattice with action of G1
LG1 := LatticeWithGram(G1, Gram); 
//coinvariant sublattice of LG1 under the action of G1 
//the above, now treated as an abstract lattice
temp := Gperb(LG1, GInvLat(LG1)); 
LG1co := LatticeWithGram(Group(temp), GramMatrix(temp));
Append(~group_names, "$L_2(11)$");
Append(~invariant_lattices, lattices1);
Append(~coinvariant_lattices, LG1co);
\end{lstlisting}

The list \texttt{group\_names}  contains the names of the 15 maximal groups in the tex format like \texttt{"\$L\_2(11)\$"}. The entries of the list \texttt{invariant\_lattices} are the lists of the possible invariant lattices $L^{\widetilde{G}}$ according to the \cite[Table 9]{HM}, in the case of $L_2(11)$ they are stored in the \texttt{lattices1} which then becomes an entry of \texttt{invariant\_lattices}. Now \texttt{G1} is a 24-dimensional representation of $L_2(11)$ as described in the auxiliary computational files for \cite{HM}. \texttt{LG1} is the coinvariant sublattice of the Leech lattice under the action of  \texttt{G1}, \texttt{Gperb} is the function \texttt{perb} computing the orthogonal complement of a sublattice inside the Leech lattice also to be found in the auxiliary files to \cite{HM} modified to work on a \texttt{GLattice} (i.~e. a lattice with a defined group action) and the function \texttt{GInvtLat} (defined in the auxiliary files) computes the invariant sublattice of a \texttt{GLattice} \texttt{L} under its associated group action via simple linear algebra.

The final result of our computations is the list \texttt{tuples} whose entries correspond to the rows in Table \ref{table:1}. We obtain them by feeding the entries of the three lists described above to the function \texttt{ExtendableIsometries}.

\begin{lstlisting}[language=MAGMA]
tuples := [];
for i in [1..15] do
	for t in ExtendableIsometries(
	    invariant_lattices[i],
	    coinvariant_lattices[i],
	    group_names[i]
	) do 
		Append(~tuples, t);
	end for;
end for;
\end{lstlisting}

Let us introduce some auxiliary functions (the implementation is simple enough to skip it an interested reader is invited to check attached file). The body of 
\begin{lstlisting}[language=MAGMA]
GoodIsometries := function(L)
    /*
        takes a 3-dimensional positive definite lattice L
        and returns the list its good isometries 
        (as defined in the paper) 
    */
end function;

Divisibility := function (h, Ld, im, proj)	
	/*	
    	h - a vector from L^G whose divisibility
    	in the K3^[2] lattice we want to find; 
    	im - an image of D(L_G) in D(L^G);
    	proj - a projection form the dual (L^G)*;
    	to the discriminant group D(L^G);
    	Ld - a dual lattice to L^G
	*/
end function;
\end{lstlisting}
Then, the body of \texttt{ExtendableIsometries} is as follows:
\begin{lstlisting}[language=MAGMA]
ExtendableIsometries := function (lattices, L_coinvariant, GName) 
    //Autdisc defined in aux. files of [HM19]
    //it returns information on the discriminant group
    //and its isometries
    Isos_DL_coinv, /* multiple other variables */ 
	    := Autdisc(L_coinvariant);
    BaseG := Group(L_coinvariant);

	list_of_tuples := [];
	extendable_isos := [];
	gens := [Matrix(g) : g in Generators(BaseG)];
	for lat in lattices do 
	    //invariant sublattice  of BaseG in the K3^[2] lattice
		L_inv := LatticeWithGram(
	        ChangeRing(
		        GramMatrix(lat),
		        Integers()
		    )
		);
		Isos_DL_inv, /* multiple other variables */ 
		:= Autdisc(L_inv);
		Iso_emb := IsometricEmbeddings(
		    DL_coinv, q_L_coinv, DL_inv, q_L_inv
		);
		isos := GoodIsometries(L_inv);
		extendable_by_gamma := [];
		for gamma in Iso_emb do
			im := Image(gamma);
			/*
			below we check whether the discriminant form
			has the apropiate value
			*/
			square := 0; //dummy value
			for x in DL_inv do
				if x in im then
					continue;
				end if;
				orthogonal := true; 
				for y in im do
					product := (q_L_inv(x+y) - q_L_inv(x) - q_L_inv(y))/2;
					if product in Integers() then
						continue;
					else
						orthogonal := false;
						break;
					end if;
				end for;
				if orthogonal then
					square := q_L_inv(x);
					break;
				end if;
			end for; 
			if square ne 3/2 then
				continue;
			end if;
			gamma_extendable := [];
			for iso in isos do
			    //h will be the polarization fixed by the group
				h := Basis(
				    InvariantLat(L_inv, MatrixGroup<3,
				        IntegerRing() | [iso]>)
				    )[1];
				//CheckGlueable will check if our good isometry
				//can be extended to an isometry of the whole lattice
				//and return true and the function on the coinvariant 
				//more on that later
				bo, f_coinv := CheckGlueable(/* many a variable */);
				if not bo then
					continue;
				end if;
				TranscendentalLattice := function(L, h) 
				    /*
				        computes the complement of h in L
				        which will be the transcendental latiice
					*/
				end function;
				T := TranscendentalLattice(L_inv, h);
				//L_inv_d is the dual of L_inv (defined by Autdisc earlier)
				//we treat h as an element of it to calcualte
				//the divisibility
				h1 := L_inv_d ! h;
				h_div := Divisibility(h1, L_inv_d, im, pi_L_inv);
				//below we check if there exists a vector
				//of divisibility 2 in the transcendental lattice
				//if there is no corresponding manifold
				//can be a Hilbert scheme of a K3 surface
				B_Td := [L_inv_d!t : t in Basis(T)];
				T_div := 1;
				if (
				    Divisibility(B_Td[1], L_inv_d, im, pi_L_inv) eq 1 and
					Divisibility(B_Td[2], L_inv_d, im, pi_L_inv) eq 1 and 
					Divisibility(B_Td[1] + B_Td[2], L_inv_d, im, pi_L_inv) 
					eq 1
				) then
					T_div := 1;
				else
					T_div := 2;
				end if;
				T := GramMatrix(T);
				if T[1][2] le 0 then
					T[1][2] *:= -1;
					T[2][1] *:= -1;
				end if;
				tuple := [* *];
				Append(~tuple, (h, h));
				/* we gather all the relevant information in tuples*/


				Append(~tuple, GramMatrix(L_inv));
				duplicate := false;
				for tu in gamma_extendable do
					/*
					    we check for duplicates by comparing values
					*/
				end if;
			end for;
			Append(~extendable_by_gamma, gamma_extendable);
		end for;
		for tuples in extendable_by_gamma do
			for old_tuple in tuples do
			    /* here we construct the tuples
			        containing the info from the table
			        as well as the generators of 
			        the overgroup of BaseG found
			        (they can be found in a separate file)
			        and check for duplicates
			     */
				Append(~extendable_isos, tuple);
			end for;
		end for;
	end for;
	return extendable_isos;
end function;
\end{lstlisting}
Now to conclude, we present the body of the \texttt{CheckGlueable} function which actually checks the conditions from Lemma \ref{extension_condition}:

\begin{lstlisting}[language=MAGMA]
CheckGlueable := function(/*many a variable*/)
    /*
        the use of permutation representations below may appear contrived
        but was probably used by Hohn and Mason for computational efficiency
        
        L0 - the coinvariant sublattice lattice in our setting;
        L - the invarinat sublattice;
        DL0 and DL - respective determinant groups;
        gamma : DL0 -> DL - a gluing isomorphism;
        iso - the isomoephism of L we wish to extend;
        Isos_DL - permutation group representation of isometries of DL;
        Isos_DL0 - ditto for DL;
	    Im_Isos_DL - permutation group representation of
	    the group of isometries of DL induced by isometries of L;
	    phi_LO - a function from automorphisms of DL0 to Isos_DL0
	    pi_L - the projection from the dual of L onto DL
        g_L0 - a map from isometries of L0 to Isos_DL0
    */
	ordiso := Order(iso);
	f_L := ChangeRing(iso, RationalField());

	gens := Generators(DL0);
	//below we check the condition 1. from the lemma 
	im := Image(gamma);
	for g in gens do
		if not pi_L((gamma(g) @@ pi_L) * f_L) in im then
			return false, _;
		end if;
	end for;
	//now we proceed the condition 2.
	f_DL0 := hom<DL0 -> DL0 |
	    [<x, pi_L((gamma(x) @@ pi_L) * f_L) @@ gamma>:
	    x in Generators(DL0)]
	>;
	ff := (AutomorphismGroup(DL0)!f_DL0)@phi_L0;
	if ff in Im_Isos_DL0 then
		ff := Im_Isos_DL0!ff;
	else
		return false, _;
	end if;
	//below we define the isometry of the coinvariant lattice
	//which can be glued with the isometry of L we started with
	f_L0 := ff @@ g_L0;
	//now we try to find f a possible //different preimage of ff 
	//with higher order relative to G
	//knowing that such f must normalize G
	BaseG := ChangeRing(Group(L0), Rationals());
	flag_order := false;
	for i in [1..Floor(ordiso/2)] do
		if f_L0^i in BaseG then
			flag_order eq true;
			break;
		end if;
	end for;
	if flag_order then
		for f in Isos_L0 do
			if not f^ordiso in BaseG then
				continue;
			end if;
			flag_order2 := false;
			for i in [1..Floor(ordiso/2)] do
				if f_L0^i in BaseG then
					flag_order2 := true;
					break;
				end if;
			end for;
			if flag_order2 then
				continue;
			end if;
			normalizes := true;
			for g in gens do
				if not f*g*f^(ordiso - 1) in BaseG then
					normalizes := false;
					break;
				end if;
			end for;
			if normalizes and g_L0(f) eq ff then
				f_L0 := f;
			end if;
		end for;
	end if;
	
	//we return true and f_L0 as an inter 20x20 matrix
	return true, ChangeRing(Matrix(f_L0), Integers());	
end function;

\end{lstlisting}

\subsection{No lines in a quartic}\label{no_lines_code}
The following \cite{M2} code checks that there are no lines in the two quartics of interest to us in Section \ref{squares}. 

We work over three rings $R = \mathbb{Q}[x_0, x_1, x_2, x_3]$, $P_1 = \mathbb{Q}[a,b,c,d]$, and $P_2 = P_1[s,t]$. The polynomial $p$ defining the quartic we are interested in lies in $R$. We consider a line in the three-dimensional projective space given by a parametrization, e.g.\ $(s, t, as + bt, cs + dt)$ where we treat $a,b,c,d$ as constants and $s, t$ as parameters. We have
\begin{align*}
    p(s, t, as + bt, cs + dt) = \sum_{\alpha + \beta =  \deg p} p_{\alpha, \beta}(a,b,c,d)s^{\alpha}t^{\beta},
\end{align*}
for some polynomials $p_{\alpha, \beta}$ of four variables. If the line is to be contained in the variety defined by $p$, all the $p_{\alpha, \beta}$ must vanish on $(a, b, c, d)$. So we check if they have a common zero. If they do not (so the radical of the ideal generated by them is not proper or is the irrelevant ideal), there is no line of the form $(s, t, as + bt, cs + dt)$ for the constants $a,b,c,d$ and the parameters $s, t$. After checking over every possible parametrization, we determine there are indeed no lines in the variety $p^{-1}(0)$.

\begin{lstlisting}[language=Macaulay2]
    R = QQ[x0, x1, x2, x3];
    P1 = QQ[a,b,c,d];
    P2 = P1[s,t];
    p1 = x0^3*x1 + x1^3*x2 + x2^3*x0 + x3^4;
    p2 = x0^4 + x1^4 + x2^4 + x3^4 + 12*x0*x1*x2*x3;
    f1 = map (P2, R, {s, t, a * s + b * t, c * s + d * t});
    f2 = map (P2, R, {s, a * s + b * t, t, c * s + d * t});
    f3 = map (P2, R, {s, a * s + b * t, c * s + d * t, t});
    f4 = map (P2, R, {a * s + b * t, s, t, c * s + d * t});
    f5 = map (P2, R, {a * s + b * t, s, c * s + d * t, t});
    f6 = map (P2, R, {a * s + b * t, c * s + d * t, s, t});
    listf = {f1, f2, f3, f4, f5, f6};
    for f in listf do (
    	print radical ideal new List from 
    	((coefficients f(p1))_1)_0
    )
    for f in listf do (
    	print radical ideal new List from
    	((coefficients f(p2))_1)_0
    )
\end{lstlisting}

\section{15 maximal groups}\label{HM_table}

 As always, notation for groups is the same as in \cite{HM} and \cite{Wilson1985ATLASOF}.

\footnotesize
\begin{center}

\begin{longtable}{ |c|c|c| } 

\hline\hline

\# & $\widetilde{G}$ & $L^{\widetilde{G}}$ \\ 

\hline\hline

\rnumber. & $L_2(11)$ &  $\begin{pmatrix}
2 & 1 & 0 \\
1 & 6 & 0 \\
0 & 0 & 22
\end{pmatrix}$, $\begin{pmatrix}
6 & 2 & 2 \\
2 & 8 & -3 \\
2 & -3 & 8
\end{pmatrix}$ \\
\hline

\rnumber. &  $L_3(4)$ & $\begin{pmatrix}
2 & 0 & 0 \\
0 & 10 & 4 \\
0 & 4 & 10
\end{pmatrix}$, $\begin{pmatrix}
4 & 2 & 0 \\
2 & 4 & 0 \\
0 & 0 & 14
\end{pmatrix}$ \\ 
\hline

\rnumber. & $A_7$ & $\begin{pmatrix}
2 & 1 & 0 \\
1 & 2 & 0 \\
0 & 0 & 70
\end{pmatrix}$, $\begin{pmatrix}
2 & 0 & 1 \\
0 & 6 & 0 \\
1 & 0 & 18
\end{pmatrix}$, $\begin{pmatrix}
4 & 2 & 1 \\
2 & 6 & 3 \\
1 & 3 & 12
\end{pmatrix}$, $\begin{pmatrix}
6 & 3 & 1 \\
3 & 6 & 1 \\
1 & 1 & 8
\end{pmatrix}$
\\ 
\hline
\rnumber. & $\mathbb{Z}_2^3:L_2(7)$ & $\begin{pmatrix}
4 & 0 & 0 \\
0 & 6 & 2 \\
0 & 2 & 10
\end{pmatrix}$
\\ 
\hline

\rnumber. & $\mathbb{Z}_2 \times L_2(7)$ & $\begin{pmatrix}
2 & 0 & 0 \\
0 & 14 & 0 \\
0 & 0 & 14
\end{pmatrix}$, $\begin{pmatrix}
4 & 2 & 0 \\
2 & 8 & 0 \\
0 & 0 & 14
\end{pmatrix}$ \\ 
\hline

\rnumber. & $\mathbb{Z}_2:A_6$ & $\begin{pmatrix}
2 & 0 & 0 \\
0 & 4 & 0 \\
0 & 0 & 24
\end{pmatrix}$, $\begin{pmatrix}
4 & 0 & 0 \\
0 & 6 & 0 \\
0 & 0 & 8
\end{pmatrix}$ \\
\hline

\rnumber. & $\mathbb{Z}_2^4:S_5$  & $\begin{pmatrix}
2 & 0 & 0 \\
0 & 4 & 0 \\
0 & 0 & 40
\end{pmatrix}$, $\begin{pmatrix}
4 & 0 & 0 \\
0 & 8 & 0 \\
0 & 0 & 10
\end{pmatrix}$ \\ 
\hline

\rnumber. & $S_6$ & $\begin{pmatrix}
4 & 2 & 0 \\
2 & 4 & 0 \\
0 & 0 & 30
\end{pmatrix}$\\
\hline

\rnumber. & $M_{10}$ & $\begin{pmatrix}
2 & 0 & 0 \\
0 & 4 & 0 \\
0 & 0 & 30
\end{pmatrix}$, $\begin{pmatrix}
4 & 2 & 0 \\
2 & 6 & 0 \\
0 & 0 & 12
\end{pmatrix}$ \\ 
\hline

\rnumber. & $(\mathbb{Z}_3 \times A_5):\mathbb{Z}_2$ & $\begin{pmatrix}
4 & 1 & 0 \\
1 & 4 & 0 \\
0 & 0 & 30
\end{pmatrix}$, $\begin{pmatrix}
6 & 0 & 0 \\
0 & 10 & 5 \\
0 & 5 & 10
\end{pmatrix}$ \\ 
\hline

\rnumber. & $Q(\mathbb{Z}_3^2:\mathbb{Z}_2)$ & $\begin{pmatrix}
6 & 2 & 2 \\
2 & 6 & -2 \\
2 & -2 & 14
\end{pmatrix}$ \\ 
\hline

\rnumber. & $\mathbb{Z}_2^4:(S_3 \times S_3)$ & $\begin{pmatrix}
4 & 0 & 0 \\
0 & 6 & 0 \\
0 & 0 & 24
\end{pmatrix}$ \\ 
\hline

\rnumber. & $\mathbb{Z}_3^2:QD_{16}$ & $\begin{pmatrix}
4 & 2 & 0 \\
2 & 10 & 0 \\
0 & 0 & 12
\end{pmatrix}$ \\ 
\hline

\rnumber. & $3^{1+4} : 2.2^2$ & $\begin{pmatrix}
6 & 0 & 0 \\
0 & 6 & 0 \\
0 & 0 & 6
\end{pmatrix}$ \\ 
\hline

\rnumber. & $3^4 : A_6$ & $\begin{pmatrix}
6 & 3 & 0 \\
3 & 6 & 0 \\
0 & 0 & 6
\end{pmatrix}$ \\ 
\hline

\caption{15 maximal groups acting symplectically on fourfolds of type K3$^{[2]}$ and the invariant sublattices}
\label{table:2}
\end{longtable}
\end{center}
\normalsize

\bibliographystyle{halpha-abbrv}
\bibliography{references}

\begin{thebibliography}{MAGMA}
\expandafter\ifx\csname url\endcsname\relax
  \def\url#1{\texttt{#1}}\fi
\expandafter\ifx\csname doi\endcsname\relax
  \def\doi#1{\burlalt{doi:#1}{http://dx.doi.org/#1}}\fi
\expandafter\ifx\csname urlprefix\endcsname\relax\def\urlprefix{URL }\fi
\expandafter\ifx\csname href\endcsname\relax
  \def\href#1#2{#2}\fi
\expandafter\ifx\csname burlalt\endcsname\relax
  \def\burlalt#1#2{\href{#2}{#1}}\fi

\bibitem[Adl78]{Adler}
A.~Adler.
\newblock On the automorphism group of a certain cubic threefold.
\newblock {\em American Journal of Mathematics}, 100(6):1275--1280, 1978.
\newblock \urlprefix\url{http://www.jstor.org/stable/2373973}.

\bibitem[BD85]{BD}
A.~Beauville and R.~Donagi.
\newblock La vari\'et\'e des droites d'une hypersurface cubique de dimension 4.
\newblock {\em C. R. Acad. Sci. Paris Ser. I Math.}, 301(14):703 -- 706, 1985.

\bibitem[Bea83]{Beauville}
A.~Beauville.
\newblock {Variet\'es K\"ahleri\`ennes dont la premi\`ere classe de Chern est
  nulle}.
\newblock {\em Journal of Differential Geometry}, 18(4):755 -- 782, 1983.
\newblock \doi{10.4310/jdg/1214438181}.

\bibitem[BH21]{BH}
S.~Brandhorst and K.~Hashimoto.
\newblock Extensions of maximal symplectic actions on {K3} surfaces.
\newblock {\em Annales Henri Lebesgue}, 4:785--809, 2021.
\newblock \doi{10.5802/ahl.88}.

\bibitem[BS21]{BS}
C.~Bonnaf\'e and A.~Sarti.
\newblock {K3 surfaces with maximal finite automorphism groups containing
  $M_{20}$}.
\newblock {\em Annales de l'Institut Fourier}, 2021.

\bibitem[BW22]{BW}
S.~Billi and T.~Wawak.
\newblock {Double EPW-sextics with actions of $\mathcal{A}_7$ and irrational GM
  threefolds}.
\newblock Available at {\url{https://arxiv.org/abs/2207.00833}}, 2022.

\bibitem[CS88]{sphere_pack}
J.~Conway and N.~Sloane.
\newblock {\em Sphere Packings, Lattices and Groups}, volume 290.
\newblock Springer New York, NY, 01 1988.
\newblock \doi{10.1007/978-1-4757-2016-7}.

\bibitem[Deb18]{D}
O.~Debarre.
\newblock Hyperkähler manifolds.
\newblock \url{https://arxiv.org/abs/1810.02087}, 2018.
\newblock \doi{10.48550/ARXIV.1810.02087}.

\bibitem[DV09]{DV}
O.~Debarre and C.~Voisin.
\newblock Hyper-{K}\"ahler fourfolds and {G}rassmann geometry.
\newblock {\em Journal für die reine und angewandte Mathematik (Crelles
  Journal)}, 2010, 04 2009.
\newblock \doi{10.1515/CRELLE.2010.089}.

\bibitem[Fu13]{Fu}
L.~Fu.
\newblock {Classification of polarized symplectic automorphisms of Fano
  varieties of cubic fourfolds}.
\newblock {\em Glasgow Mathematical Journal}, -1, 03 2013.
\newblock \doi{10.1017/S001708951500004X}.

\bibitem[GAP21]{GAP4}
The GAP~Group.
\newblock {\em {GAP -- Groups, Algorithms, and Programming, Version 4.11.1}},
  2021.
\newblock \urlprefix\url{{https://www.gap-system.org}}.

\bibitem[HM19]{HM}
G.~H\"ohn and G.~Mason.
\newblock Finite groups of symplectic automorphisms of hyperk\"ahler manifolds
  of type {K}3$^{[2]}$.
\newblock {\em Bulletin of the Institute of Mathematics, Academia Sinica (New
  Series)}, 14:189--26, 2019.

\bibitem[Huy99]{Huybrechts2}
D.~Huybrechts.
\newblock Compact hyperkaehler manifolds: Basic results.
\newblock {\em Inventiones Mathematicae}, 135:63--113, 01 1999.
\newblock \doi{10.1007/s002220050280}.

\bibitem[Huy16]{huybrechts_2016}
D.~Huybrechts.
\newblock {\em Lectures on K3 Surfaces}.
\newblock Cambridge Studies in Advanced Mathematics. Cambridge University
  Press, 2016.
\newblock \doi{10.1017/CBO9781316594193}.

\bibitem[Kon99]{Ko}
S.~Kondo.
\newblock {The maximum order of finite groups of automorphisms of K3 surfaces}.
\newblock {\em American Journal of Mathematics}, 121:1245--1252, 12 1999.
\newblock \doi{10.1353/ajm.1999.0040}.

\bibitem[Loo09]{Looijenga}
E.~Looijenga.
\newblock The period map for cubic fourfolds.
\newblock {\em Inventiones mathematicae}, 177, 07 2009.
\newblock \doi{10.1007/s00222-009-0178-6}.

\bibitem[LZ19]{LazaZheng}
R.~Laza and Z.~Zheng.
\newblock Automorphisms and periods of cubic fourfolds.
\newblock 2019.
\newblock \doi{10.48550/ARXIV.1905.11547}.

\bibitem[M2]{M2}
D.~R. Grayson and M.~E. Stillman.
\newblock Macaulay2, a software system for research in algebraic geometry.
\newblock Available at {\url{http://www.math.uiuc.edu/Macaulay2/}}.

\bibitem[MAGMA]{MAGMA}
W.~Bosma, J.~Cannon, and C.~Playoust.
\newblock The {M}agma algebra system. {I}. {T}he user language.
\newblock {\em J. Symbolic Comput.}, 24(3-4):235--265, 1997.
\newblock \doi{10.1006/jsco.1996.0125}.
\newblock Computational algebra and number theory (London, 1993).

\bibitem[Mar11]{Markman}
E.~Markman.
\newblock {A survey of Torelli and monodromy results for holomorphic-symplectic
  varieties}.
\newblock in \textit{Complex and differential geometry}, 2011.

\bibitem[Mon13]{Mongardi}
G.~Mongardi.
\newblock {Automorphisms of hyperk\"ahler manifolds (PhD thesis)}.
\newblock \url{https://arxiv.org/abs/1303.4670}, 03 2013.

\bibitem[Muk88]{Mu}
S.~Mukai.
\newblock Finite groups of automorphisms of {K}3 surfaces and the {M}athieu
  group.
\newblock {\em Inventiones Mathematicae}, 94:183--221, 02 1988.
\newblock \doi{10.1007/BF01394352}.

\bibitem[Nik79]{Nikulin}
V.~V. Nikulin.
\newblock Integer symmetric bilinear forms and some of their geometric
  applications.
\newblock {\em Izv. Akad. Nauk SSSR Ser. Mat.}, 43:111--177, 01 1979.

\bibitem[O'G06]{o2006irreducible}
K.~G. O'Grady.
\newblock {Irreducible symplectic 4-folds and Eisenbud-Popescu-Walter sextics}.
\newblock {\em Duke Mathematical Journal}, 134(1):99--137, 2006.

\bibitem[Son21]{Song}
J.~Song.
\newblock {A special Debarre-Voisin fourfold}.
\newblock Available at {\url{https://arxiv.org/abs/2106.13287}}, 2021.

\bibitem[WCN85]{Wilson1985ATLASOF}
R.~A. Wilson, J.~H. Conway, and S.~P. Norton.
\newblock {ATLAS of Finite Groups}.
\newblock 1985.

\end{thebibliography}
\end{document}